\documentclass{amsart}
\usepackage{geometry} 
\usepackage{amsmath,amsthm,amssymb,amsfonts}
\usepackage{xcolor}
\usepackage{mathrsfs}
\usepackage{palatino}
\usepackage{verbatim}
\usepackage{float}
\usepackage{graphicx}
\usepackage{tikz-cd}
\usepackage{marvosym}
\usepackage{hyperref}
\usepackage{ytableau}
\usepackage{multicol}
\usepackage{enumerate}
\usepackage{multirow}

\newcommand{\N}{\mathbf{N}}
\newcommand{\Z}{\mathbf{Z}}

\newcommand{\Q}{\mathbf{Q}}
\newcommand{\C}{\mathbf{C}}
\newcommand{\F}{\mathbf{F}}

\DeclareMathOperator{\Gal}{Gal}

\DeclareMathOperator{\Sp}{Sp}
\DeclareMathOperator{\GSp}{GSp}
\DeclareMathOperator{\SL}{SL}
\DeclareMathOperator{\GL}{GL}

\DeclareMathOperator{\Sym}{Sym}
\DeclareMathOperator{\tr}{tr}
\DeclareMathOperator{\Frob}{Frob}
\DeclareMathOperator{\Aut}{Aut}

\DeclareMathOperator{\an}{an}

\newcommand{\ind}{\mathbf{1}}

\DeclareMathOperator{\Stab}{Stab}
\DeclareMathOperator{\Orb}{Orb}

\DeclareMathOperator{\Var}{Var}

\theoremstyle{definition}
\newtheorem{thm}{Theorem}[section]
\newtheorem{defn}[thm]{Definition}

\newtheorem{cor}[thm]{Corollary}
\newtheorem{lem}[thm]{Lemma}
\newtheorem{prop}[thm]{Proposition}
\newtheorem{conj}[thm]{Conjecture}

\newtheorem*{thm:universal-surface}{Theorem~\ref{thm:universal-surface}}
\newtheorem*{thm:low-degree-computation}{Theorem~\ref{thm:low-degree-computation}}
\newtheorem*{thm:low-degree-sym-n}{Theorem~\ref{thm:low-degree-sym-n}}
\newtheorem*{cor:avg-fq-points}{Corollary~\ref{cor:avg-fq-points}}
\newtheorem*{cor:nth-fiber-asymptotic-in-q}{Corollary~\ref{cor:nth-fiber-asymptotic-in-q}}
\newtheorem*{cor:sym-asymptotic-in-q}{Corollary~\ref{cor:sym-asymptotic-in-q}}
\newtheorem*{cor:variance}{Corollary~\ref{cor:variance}}

\theoremstyle{remark}
\newtheorem{rmk}[thm]{Remark}

\begin{document}
\title{Cohomology of the Universal Abelian Surface with Applications to Arithmetic Statistics}
\author{Seraphina Eun Bi Lee}
\address{Department of Mathematics, University of Chicago}
\email{seraphinalee@uchicago.edu}

\begin{abstract}
The moduli stack $\mathcal A_2$ of principally polarized abelian surfaces comes equipped with the universal abelian surface $\pi: \mathcal X_2 \to \mathcal A_2$. The fiber of $\pi$ over a point corresponding to an abelian surface $A$ in $\mathcal A_2$ is $A$ itself. We determine the $\ell$-adic cohomology of $\mathcal X_2$ as a Galois representation. Similarly, we consider the bundles $\mathcal X_2^n \to \mathcal A_2$ and $\mathcal X_2^{\Sym(n)} \to \mathcal A_2$ for all $n \geq 1$, where the fiber over a point corresponding to an abelian surface $A$ is $A^n$ and $\Sym^n A$ respectively. We describe how to compute the $\ell$-adic cohomology of $\mathcal X_2^n$ and $\mathcal X_2^{\Sym(n)}$ and explicitly calculate it in low degrees for all $n$ and in all degrees for $n = 2$. These results yield new information regarding the arithmetic statistics on abelian surfaces, including an exact calculation of the expected value and variance as well as asymptotics for higher moments of the number of $\F_q$-points.
\end{abstract}
\maketitle
\tableofcontents
\section{Introduction}\label{sec:introduction}
An \emph{abelian surface} is an abelian variety of dimension $2$. Over $\C$, all abelian surfaces are isomorphic to $\C^2/L$ for some lattice $L$ with real rank $4$. The fine moduli stack $\mathcal A_2$ of principally polarized abelian surfaces is a smooth Deligne--Mumford stack defined over $\Z$. It comes equipped with a universal bundle $\mathcal X_2 \to \mathcal A_2$. The fiber over the point corresponding to an abelian surface $A$ in $\mathcal A_2$ is $A$ itself. Using the projection map $\mathcal X_2 \to \mathcal A_2$, we can take $n$th fiber powers $\mathcal X_2^n$ of $\mathcal X_2$ over $\mathcal A_2$, which has the $n$th power $A^n$ of an abelian surface over the corresponding point in $\mathcal A_2$. Since each $A^n$ has an action of $S_n$ permuting the coordinates (which is not a free action), taking the quotient $\mathcal X_2^n/S_n$ gives a new stack $\mathcal X_2^{\Sym(n)}$, which has $\Sym^n A$ as a fiber over the point corresponding to $A$ in $\mathcal A_2$. 

Our main theorems are the computations of the $\ell$-adic cohomology of the universal abelian surface and related spaces as Galois representations (up to semi-simplification). From now on, all cohomology will denote $\ell$-adic cohomology and we drop the subscripts to write $H^*(\mathcal X; \mathbf V)$ in place of both $H^*_{\text{\'et}}(\mathcal X_{\overline \Q}; \mathbf V)$ and $H^*_{\text{\'et}}(\mathcal X_{\overline \F_q}; \mathbf V)$ for any $\ell$-adic local system $\mathbf V$ on $\mathcal X$ with $\ell$ coprime to $q$. (See Remark \ref{rmk:etale-Q-Fq} for details justifying this notation.)
\begin{thm}\label{thm:universal-surface}
The cohomology of the universal abelian surface $\mathcal X_2$ is given by
\[
  H^k(\mathcal X_2; \Q_\ell) = \begin{cases}
  \Q_\ell & k = 0 \\
  0 & k = 1, 3, k \geq 7\\
  2\Q_\ell(-1) & k = 2 \\
  2\Q_\ell(-2) & k = 4 \\
  \Q_\ell(-5) & k = 5 \\
  \Q_\ell(-3) & k = 6 \\
  \end{cases}
\]
up to semi-simplification, where $\Q_\ell = \Q_\ell(0)$ is the trivial Galois representation, $\Q_\ell(1)$ is the $\ell$-adic cyclotomic character, and $\Q_\ell(-1)$ is its dual. For all $n \in \N$, $\Q_\ell(n)$ is the $n$th tensor power of $\Q_\ell(1)$ and $\Q_\ell(-n)$ is the $n$th tensor power of $\Q_\ell(-1)$. For all $n \in \Z$, $\Q_\ell(n) \cong \Q_\ell$ as a $\Q_\ell$-vector space. Denote $\Q_\ell(n)^{\oplus m}$ by $m \Q_\ell(n)$.
\end{thm}

Applying similar techniques gives the cohomology of the $n$th fiber product of $\mathcal X_2$ in low degrees. The following theorem applies these techniques to explicitly compute $H^k(\mathcal X_2^n; \Q_\ell)$ for $0 \leq k \leq 5$ for all $n \geq 1$. Here and in the rest of the paper, we use the convention that $\binom nm = 0$ if $n < m$. 
\begin{thm}\label{thm:low-degree-computation}
For all $n \geq 1$, the cohomology of the universal $n$th fiber product of abelian surfaces is
\begin{align*}
H^k(\mathcal X_2^n; \Q_\ell) &= \begin{cases}
  \Q_\ell & k = 0 \\
  0 & k = 1, 3 \\
  \left(\binom{n+1}{2} + 1\right) \Q_\ell(-1) & k = 2 \\
  \left(\frac{n(n+1)(n^2+n+2)}{8} + \binom{n+1}{2}\right)\Q_\ell(-2) & k = 4\\
  \binom{n+1}{2} \Q_\ell(-5) \oplus \binom n2 \Q_\ell(-4) & k = 5
\end{cases}
\end{align*}
up to semi-simplification.
\end{thm}

The cohomology of the universal $n$th symmetric power of abelian surfaces \emph{stabilizes} as $n$ increases, by which we mean that the cohomology $H^k(\mathcal X^{\Sym(n)}; \Q_\ell)$ is independent of $n$ for $n$ large enough compared to the degree $k$. As with the cohomology of the $n$th fiber product of $\mathcal X_2$, we explicitly compute $H^k(\mathcal X_2^{\Sym(n)};\Q_\ell)$ for $0\leq k \leq 5$ and all $n$ large enough compared to $k$.
\begin{thm}\label{thm:low-degree-sym-n}
For all $n \geq k$ for $k$ even and for all $n \geq k-1$ for $k$ odd,
\[
  H^k(\mathcal X_2^{\Sym(n)}; \Q_\ell) = \begin{cases}
  \Q_\ell & k = 0 \\
  0 & k = 1, 3 \\
  3\Q_\ell(-1) & k = 2 \\
  9\Q_\ell(-2) & k = 4 \\
  2\Q_\ell(-5) & k = 5
  \end{cases}
\]
up to semi-simplification.
\end{thm}

The proofs of these theorems use the Leray spectral sequence of the morphisms $\pi: \mathcal X \to \mathcal A_2$, with $\mathcal X = \mathcal X_2$, $\mathcal X_2^n$, and $\mathcal X_2^{\Sym(n)}$ respectively. The spectral sequence takes as input the cohomology of local systems of $\mathcal A_2$, which has been computed by Petersen in \cite{petersen}. Then it still remains to determine the local systems involved in the latter two cases, converting this problem about cohomology into a series of problems about the representation theory of $\Sp(4)$. For $\mathcal X_2^n$, we give recursive formulas (in $n$) for the relevant local systems in Subsection \ref{sec:fiber-powers-n}. For $\mathcal X_2^{\Sym(n)}$, we show that the local systems $R^k\pi_*\Q_\ell$ stabilize for $n \geq k$. In both cases, we use these facts to prove Theorems \ref{thm:low-degree-computation} and \ref{thm:low-degree-sym-n}. 

The cohomology in higher degrees is quite involved to determine for general $n\geq 0$ and $k \geq 0$ and involve Galois representations attached to certain (Siegel) modular forms. However, the methods of this paper give a finite computation for the relevant local systems for each fixed $n$. This means that given enough information about the inputs to Petersen's theorem (\cite[Theorem 2.1]{petersen}), it is possible to compute the cohomology groups for larger values $k$ and $n$ using the results of this paper. We work out the case $n = 2$ completely -- see Theorems \ref{thm:2-power-universal-surface} and \ref{thm:sym-2} for the cohomology of $\mathcal X_2^2$ and $\mathcal X_2^{\Sym(2)}$ respectively.

\bigskip
\noindent
{\bf{Arithmetic Statistics.}}
We fix throughout a finite field $\F_q$. The Weil conjectures give bounds on the number of $\F_q$-points on any projective variety over $\F_q$. Applied to an abelian surface $A$ they assert that
\[
  \#A(\F_q) = q^2 + a_3q^{3/2} + a_2q + a_1q^{1/2} + 1
\]
where $a_i$ are some sums of $n_i$ roots of unity with $n_i = 4, 6, 4$ for $i = 1, 2, 3$ respectively. A simple corollary is
\[
  \lvert \#A(\F_q) - (q^2 + 1) \rvert \leq 4q^{3/2} + 6q + 4q^{1/2},
\]
constraining the possible values that $\#A(\F_q)$ can take. The exact set of possible values of $\#A(\F_q)$ is given by \emph{Honda--Tate theory}, which yields a bijection between isogeny classes of simple abelian varieties (of all dimensions) over $\F_q$ and Weil $q$-polynomials. In particular, for a Weil $q$-polynomial $f$, there is some abelian variety $V$ such that the characteristic polynomial $f_V$ of the Frobenius endomorphism of $V$ is given by $f_V = f^e$ for some $e \geq 1$, for which $\# V(\F_q) = f_V(1)$. While the restriction of this bijection to simple abelian surfaces is known, we omit it for brevity and refer the reader to \cite{ruck}, \cite{waterhouse} and \cite[Section 2]{dgssst} for an overview.

Studying the counts $\# \mathcal X_2^n(\F_q)$ and $\# \mathcal X_2^{\Sym(n)}(\F_q)$ for $n \geq 1$ will give more information about the distribution of $\#A(\F_q)$ as the abelian surface $A$ varies over $\mathcal A_2(\F_q)$. Our main tool to obtain these point counts is the Grothendieck--Lefschetz--Behrend trace formula (\cite[Theorem 3.1.2]{behrend-inventiones}). A first observation through standard applications of the Weil conjectures and the trace formula is that $\#\mathcal X_2^n(\F_q) = q^d + O(q^{d - \frac 12})$ where $d = \dim \mathcal X_2^n$, because $\mathcal X_2^n$ is finitely covered by a smooth, irreducible, quasiprojective variety. However, applying the trace formula to our cohomological theorems immediately gives more precise asymptotics for $\#\mathcal X_2^n(\F_q)$ as well as new arithmetic statistics about the  number of $\F_q$-points of abelian surfaces. Below, we consider expected values of random variables on $\mathcal A_2(\F_q)$ by giving $\mathcal A_2(\F_q)$ a natural probability measure where each isomorphism class of an abelian surface $A$ has probability inversely proportional to the size of its $\F_q$-automorphism group; see Lemma \ref{lem:prob-defns} for more details.
\begin{cor}\label{cor:avg-fq-points}
The expected number of $\F_q$-points on abelian surfaces defined over $\F_q$ is
\[
  \mathbf E[\#A(\F_q)] = q^2 + q + 1 - \frac{1}{q^3 + q^2}.
\]
\end{cor}
For each prime power $q > 0$, there is a simple abelian surface $A_q$ over $\F_q$ with $\#A_q(\F_q) = q^2 + q + 1$ by the Honda--Tate correspondence for surfaces (\cite[Theorem 1.1]{ruck}) which corresponds to the case $a_3 = a_1 = 0$, $a_2 = 1$ of the Weil conjectures. Although $\mathbf E[\#A(\F_q)]$ is not realized by an abelian surface for any fixed $q$, the minimal difference between the expected value and the $\F_q$-point count of an arbitrary abelian surface goes to $0$ as $q$ increases, i.e. 
\[
  \min_{[A] \in \mathcal A_2(\F_q)} \lvert\#A(\F_q) - \mathbf E[\#A(\F_q)]\rvert \to 0 \quad\text{ as } q \to \infty.
\]

For any abelian surface $A$ and $n > 1$, $A^n(\F_q)$ denotes the set of $\F_q$-points of the $n$th power $A^n$ of $A$, which are ordered $n$-tuples of (not necessarily distinct) $\F_q$-points of $A$. The trace formula is also used to compute the exact expected value of $\#A^2(\F_q)$ and an asymptotic estimate for the expected value of $\#A^n(\F_q)$ for $n > 2$. Because $\#A^n(\F_q) = (\#A(\F_q))^n$ for any abelian surface $A$, the following corollary gives the exact second moment of the number of $\F_q$-points on abelian surfaces and asymptotic estimates on the $n$th moment for all $n \geq 3$. 
\begin{cor}\label{cor:nth-fiber-asymptotic-in-q}
The expected value of $\#A^2(\F_q)$ is
\[
  \mathbf E[\#A^2(\F_q)] = q^4 + 3q^3 + 6q^2 + 3q - \frac{5q^2 + 5q + 3}{q^3 + q^2}
\]
and for all $n \geq 1$, 
\[
  \mathbf E[\# A^n(\F_q)] = q^{2n} + \binom{n+1}{2} q^{2n-1} + \left(\frac{n(n+1)(n^2+n+2)}{8}\right) q^{2n-2} + O(q^{2n-3}).
\]
\end{cor}
Note that computing the $n$th moment for large $n$ involves representations attached to (Siegel) modular forms; therefore, the recursive formulas for local systems in the cohomological computations do not completely determine these moments. However for fixed $n \geq 0$, the trace formula does give the exact $n$th moment of $\# A(\F_q)$ if $H^k(\mathcal X^n; \Q_\ell)$ for all $k \geq 0$ are known.

For any abelian surface $A$ and $n > 1$, $\Sym^nA(\F_q)$ denotes the set of $\F_q$-points of the symmetric power $\Sym^n A$ of an abelian surface $A$, which are the unordered $n$-tuples of (not necessarily distinct) $\overline\F_q$-points of $A$ defined as an $n$-tuple over $\F_q$. This means that the $n$-tuple contains all Galois conjugates of each point of the $n$-tuple. The same methods give the exact expected value of $\#\Sym^2 A(\F_q)$ and an asymptotic estimate for the expected value of $\#\Sym^n A(\F_q)$ for $n > 2$. 
\begin{cor}\label{cor:sym-asymptotic-in-q}
The expected value of $\#\Sym^n A(\F_q)$ for $n = 2$ is 
\[
  \mathbf E[\#\Sym^2 A(\F_q)] = \frac{\#\mathcal X_2^{\Sym(2)}(\F_q)}{\#\mathcal A_2(\F_q)}= q^4 + 2q^3 + 4q^2 + 2q + 1 - \frac{3q^2 + 2q + 2}{q^3 + q^2}.
\]
For $n = 3$,
\[
  \mathbf E[\#\Sym^3 A(\F_q)] = \frac{\#\mathcal X_2^{\Sym(3)}(\F_q)}{\#\mathcal A_2(\F_q)} = q^6 + 2q^5 + O(q^4)
\]
and for all $n \geq 4$,
\[
  \mathbf E[\#\Sym^nA(\F_q)] =\frac{\#\mathcal X_2^{\Sym(n)}(\F_q)}{\#\mathcal A_2(\F_q)}= q^{2n} + 2q^{2n-1} + 7q^{2n-2} + O(q^{2n-3}).
\]
\end{cor}

Because Corollary \ref{cor:nth-fiber-asymptotic-in-q} gives the exact second moment of $\#A(\F_q)$, we also obtain the variance:
\begin{cor}\label{cor:variance}
The variance of $\#A(\F_q)$ is
\[
  \Var(\#A(\F_q)) = q^3 + 3q^2 + q - 1 - \frac{3q^2 + 3q + 1}{q^3+ q^2} - \frac{1}{(q^3 + q^2)^2}.
\]
\end{cor}

These statistics are computed in Subsection \ref{sec:glb-trace} by studying $\#\mathcal X(\F_q)$ for various stacks $\mathcal X$. All $\F_q$-point counts in this paper are weighted by the inverse of the size of their automorphism groups, as explained in Section \ref{sec:arithmetic-statistics}. 

\bigskip
\noindent
{\bf{Related work.}}
The studies of the cohomologies of the moduli space of abelian varieties and the universal abelian variety, point counts over finite fields of these varieties, and Siegel modular forms are tightly intertwined. For example, the cohomology of local systems on the moduli space of elliptic curves is known classically (e.g. the Eichler--Shimura isomorphism) and yields connections between modular forms and point counts of elliptic curves over finite fields. (See \cite{van-der-geer-abvar} and \cite{hulek--tommasi} for a survey on current developments in this area.) In the case of abelian surfaces, Faber and van der Geer first pursued this approach in \cite{faber-van-der-geer-1, faber-van-der-geer-2}, giving conjectural formulas for the class of the $\ell$-adic cohomology of local systems of $\mathcal A_2$ in the Grothendieck group of $\ell$-adic Galois representations based on computer-generated point counts over finite fields. These conjectures were proven by Weissauer (\cite{weissauer-regular}) in the case of local systems with regular highest weight, and by Petersen (\cite{petersen}) in the most general case. The connections between these ideas are implicit in our paper; we fully take advantage of the work mentioned above, in both the computations of the cohomology of the relevant spaces and the resulting arithmetic statistics.

Recent work in arithmetic statistics of abelian varieties also take a different flavor than of this paper. Honda--Tate theory has been used to determine some probabilistic data about the group structure of abelian surfaces (\cite{dgssst}), upper and lower bounds on the number of $\F_q$-points on abelian varieties (\cite{aubry--haloui--lachaud}), sizes of isogeny classes of abelian surfaces (\cite{xue--yu}), and others. Certainly, this is not the only current approach in this direction -- for example, \cite{cfhs} takes a heuristic approach to determining the probability that the number of $\F_q$-points on a genus 2 curve is prime.

On the other hand, cohomological methods have been applied to related spaces to deduce arithmetic statistical results or heuristics, such as the number of points on curves of genus $g$ (\cite{aekwz}) and the average number of points on smooth cubic surfaces (\cite{das}). For a survey in the case of counting genus $g$ curves and its connection to the cohomology of the relevant moduli spaces, see \cite{van-der-geer}.

\bigskip
\noindent
{\bf{Outline of paper.}} 
In Section \ref{sec:A2}, we give a description of the spaces of study and the cohomological tools used throughout the paper. In Section \ref{sec:universal-surface}, we prove Theorem \ref{thm:universal-surface} and in Section \ref{sec:fiber-powers}, we prove Theorem \ref{thm:low-degree-computation} and completely work out the cohomology of $\mathcal X_2^2$ as an example. In Section \ref{sec:symmetric-powers}, we carry out analogous arguments to prove Theorem \ref{thm:low-degree-sym-n} and work out the cohomology of $\mathcal X_2^{\Sym(2)}$. In Section \ref{sec:arithmetic-statistics}, we deduce new arithmetic statistics results about abelian surfaces using the previous sections, including Corollaries \ref{cor:avg-fq-points}, \ref{cor:nth-fiber-asymptotic-in-q}, \ref{cor:sym-asymptotic-in-q}, and \ref{cor:variance}.

\bigskip
\noindent
{\bf{Acknowledgements.}} 
I am deeply grateful to my advisor Benson Farb for his support and guidance throughout this project, from suggesting this problem to commenting extensively on numerous earlier drafts. I am also very grateful to Aleksander Shmakov for his continued generous help with many aspects of the project, including explaining technical background details and giving thorough comments and corrections on a previous draft. I would like to thank Eduard Looijenga, Dan Petersen, and Alexander Yong for their patient responses to many questions about the subject matter. I thank Ronno Das, Nir Gadish, and Linus Setiabrata for insightful conversations and comments on an earlier draft of this paper, as well as Linus for catching an error in that draft. I thank Frank Calegari and Jordan Ellenberg for useful comments which improved the exposition of this paper. I also thank Jeff Achter for indicating a relevant computation and Carel Faber, Nate Harman, Nat Mayer, and Philip Tosteson for helpful correspondences. Lastly, I thank the anonymous referee for their meticulous comments that vastly improved this paper.

\section{$\mathcal A_2$, its Cohomology, and Cohomological Tools}\label{sec:A2}
In this section, we describe the spaces that we study in this paper and outline the cohomological tools that will be used throughout the paper.
\subsection{Spaces of interest.}\label{sec:spaces}
Denote the moduli stack of principally polarized abelian surfaces by $\mathcal A_2$. For concreteness, we discuss an explicit construction of the set of complex points $\mathcal A_2(\C)$ of $\mathcal A_2$: let $\mathcal H_2$ be the Siegel upper half space of degree $2$ with the usual action of $\Sp(4, \Z)$,
\[
  \begin{pmatrix}
  A & B \\ C & D
  \end{pmatrix} \cdot \tau = (C\tau + D)^{-1}(A \tau + B).
\]
To each $\tau \in \mathcal H_2$, we can associate a lattice $L_\tau \subseteq \C^2$ and therefore a complex torus $A_\tau$. It turns out that $A_\tau$ comes with a natural principal polarization $H_\tau$, making $(A_\tau, H_\tau)$ into a principally polarized abelian variety. For any $\tau_1$, $\tau_2 \in \mathcal H_2$, the abelian surfaces $(A_{\tau_1}, H_{\tau_1})$ and $(A_{\tau_2}, H_{\tau_2})$ are isomorphic if and only if $\tau_1$ and $\tau_2$ are in the same $\Sp(4, \Z)$-orbit.

The action of $\Sp(4, \Z)$ on $\mathcal H_2$ is not free. For example, $-I_4$ fixes every $\tau \in \mathcal H_2$. However, the stabilizer of each point is finite. Therefore, $\mathcal A_2(\C)$ is the set of points of the orbifold $\Sp(4, \Z)\backslash\mathcal H_2$, with the underlying analytic space of $\Sp(4, \Z) \backslash \mathcal H_2$ denoted $(\mathcal A_2)_\C^{\an}$. 

The cohomology of $\mathcal A_2$ is known:
\begin{thm}[{\cite[Corollary 5.2.3]{lee--weintraub}, \cite[Section 10]{van-der-geer-abvar}}]\label{thm:a2}
\[
  H^k(\mathcal A_2; \Q_\ell) = \begin{cases}
  \Q_\ell & k = 0 \\
  \Q_\ell(-1) & k = 2 \\
  0 & \text{otherwise}.
  \end{cases} 
\]
\end{thm}

Next, we denote the universal abelian surface by $\mathcal X_2$ and give an explicit construction for the complex points $\mathcal X_2(\C)$ of $\mathcal X_2$. Take the action of $\Sp(4, \Z)\ltimes\Z^4$ on $\mathcal H_2 \times \C^2$, where $\Z^4$ acts by translation on each $L_\tau \subseteq \C^2$. Then $\mathcal X_2(\C)$ is the set of points of the orbifold $\Sp(4, \Z) \ltimes \Z^4 \backslash \mathcal H_2 \times \C^2$ with the underlying analytic space of $(\Sp(4, \Z) \ltimes \Z^4) \backslash (\mathcal H_2 \times \C^2)$ denoted $(\mathcal X_2)_\C^{\an}$. Note that the fiber of the natural projection $\mathcal X_2(\C) \to \mathcal A_2(\C)$ over a point corresponding to the surface $A$ is the set $A(\C)$ of $\C$-points of $A$ itself.

The stack $\mathcal X_2^n$ is defined to be the fiber product of $\mathcal X_2 \to \mathcal A_2$ with itself $n$ times and comes with a natural map $\mathcal X_2^n \to \mathcal A_2$. As before, the group $\Sp(4, \Z) \ltimes (\Z^4)^n$ acts on $\mathcal H_2 \times (\C^2)^n$ in the obvious way, and so the set $\mathcal X_2^n(\C)$ of the $\C$-points of $\mathcal X_2^n$ is the set of points of the orbifold $\Sp(4, \Z) \ltimes (\Z^4)^n \backslash \mathcal H_2 \times (\C^2)^n$ with the underlying analytic space $(\mathcal X_2^n)_\C^{\an}$ given by the usual quotient. The fiber of $\mathcal X_2^n(\C) \to\mathcal A_2(\C)$ over a point corresponding to the surface $A$ is the set $A^n(\C)$ of $\C$-points of the $n$th power $A^n$.

Also consider the stack $\mathcal X_2^{\Sym(n)}$. Each fiber $A^n$ of the projection morphism $\mathcal X_2^n \to \mathcal A_2$ has an action of $S_n$ permuting the coordinates, giving a stack-theoretic quotient $\mathcal X_2^{\Sym(n)} = [\mathcal X_2^n/S_n]$. The fiber of $\mathcal X_2^{\Sym(n)} \to \mathcal A_2$ of the point corresponding to the abelian surface $A$ is $\Sym^nA$, the stack quotient $[A^n/S_n]$. As usual, there is an underlying analytic space $(\mathcal X_2^{\Sym(n)})_\C^{\an} = (\mathcal X_2^n)_\C^{\an}/S_n$. 

For any $N \geq 2$ and abelian surface $A$, let $A[N]$ be the kernel of the multiplication by $N$ map on $A$. Consider the moduli stack $\mathcal A_2[N]$ of principally polarized abelian surfaces with symplectic level $N$ structure, i.e. pairs $(A, \alpha)$ where $\alpha$ is an isomorphism from $A[N]$ to a fixed symplectic module $((\Z/N\Z)^4, \langle \cdot, \cdot \rangle)$. Let $\Sp(4, \Z)[N] = \ker(\Sp(4, \Z) \to \Sp(4, \Z/N\Z))$. If $N \geq 3$, both $\mathcal A_2[N]$ and $\mathcal X_2[N]$ are quasiprojective schemes over $\Z[\frac 1N, \zeta_N]$ (\cite[Chapter IV]{faltings--chai}).

On the other hand, consider the moduli stack $\mathcal A_{2,N}$ of principally polarized abelian surfaces with principal level $N$ structure, i.e. pairs $(A, \beta)$ where $\beta: A[N] \to (\Z/N\Z)^4$ is an isomorphism. For $N \geq 3$, $\mathcal A_{2,N}$ and its universal family $\mathcal X_{2,N}$ are quasiprojective schemes over $\Z[\frac 1N]$ (\cite[Chapter I]{faltings--chai}). This shows that over $\Z[\frac 1N]$ with $N \geq 3$, the stacks $\mathcal X = \mathcal A_2$, $\mathcal X_2^n$, and $\mathcal X_2^{\Sym(n)}$ are all finite quotients of quasiprojective schemes (cf. \cite[Theorem 2.1.11]{olsson}). Over characteristic zero, quotient stacks with finite automorphism groups at every point are Deligne--Mumford stacks (\cite[Corollary 2.2]{edidin}). Over positive characteristic, quotient stacks are a priori Artin stacks with a smooth atlas. Because any base change of an \'etale morphism is \'etale, any stack $\mathcal X$ considered in this paper obtained from a stack over $\Z[\frac 1N]$ via base-change to $\F_q$ (where $N$ and $q$ are coprime) has an \'etale atlas; therefore, $\mathcal X_{\F_q}$ is a Deligne--Mumford stack.

In fact, $\mathcal A_2$ and $\mathcal X_2^n$ are complements of normal crossing divisors in smooth, proper stacks over $\Z$ (see \cite[Chapter VI]{faltings--chai}), making both $\mathcal A_2$ and $\mathcal X_2^n$ as well as its finite quotient $\mathcal X_2^{\Sym(n)}$ smooth stacks over any finite field $\F_q$. Over any field $k$, there are the following moduli interpretations of the $k$-points of these stacks. When we write an abelian surface $A$, we mean $A$ with a principal polarization.
\begin{enumerate}
  \item $[\mathcal A_2(k)]$ is the set of $k$-isomorphism classes of abelian surfaces $A$ defined over $k$. \\
  \item $[\mathcal A_2[N](k)]$ is the set of $k$-isomorphism classes of pairs $(A, \alpha)$ where $\alpha: A[N] \to (\Z/N\Z)^4$ is a symplectic isomorphism, \\
  \item $[\mathcal X_2^n(k)]$ is the set of $k$-isomorphism classes of pairs $(A, p)$ with $p \in A^n$, defined over $k$. \\
  \item $[\mathcal X_2^{\Sym(n)}(k)]$ is the set of $k$-isomorphism classes of pairs $(A, p)$ with $p \in A^n/S_n$, defined over $k$.
\end{enumerate}
\begin{rmk}
We note that for some $\{p_1, \dots, p_n\} \in (A^n/S_n)(k)$, a lift $(p_1, \dots, p_n) \in A^n$ may not be a $k$-point of $A^n$. For instance, if $\Gal(\overline k/k)$ permutes the points $p_1, \dots, p_n \in A^n$, then $\{p_1, \dots, p_n\}$ will be a $k$-point of $\Sym^n A$, but not necessarily of $A^n$. 
\end{rmk}

Although this is possibly not the most efficient framework, we will access all stacks discussed by taking quotient stacks of the respective quasiprojective schemes throughout this paper in an effort to keep the arguments as concrete as possible. 

\subsection{Local systems on $\mathcal A_2$.}\label{sec:local-systems}
Representations of $\GSp(4, \Q_\ell)$ give rise to $\ell$-adic local systems on $(\mathcal A_2)_{\Z[1/\ell]}$ (\cite[p. 238]{faltings--chai}). The local systems are considered instead as representations of $\Sp(4, \Q_\ell)$ in \cite{petersen}; we also study the underlying $\Sp(4, \Q_\ell)$-representations of local systems at various points in this paper. In this section we review the construction of local systems on $(\mathcal A_2)_{\Z[1/\ell]}$; the reader may also consult \cite[Section 4]{bfv-degree-3} or \cite[Section 4]{v-rank-1-eisenstein}.

By Weyl's construction (see \cite[Section 17.3]{fulton--harris}), all irreducible representations of $\Sp(4, \Q_\ell)$ are given in the following way: for any $a \geq b \geq 0$, there is an irreducible representation $W_{a, b}$ with highest weight $aL_1 + bL_2$, using the notation of \cite[Chapter 17]{fulton--harris}. In particular, $W_{a, b}$ is a summand of $W_{1,0}^{\otimes (a+b)}$ and is the irreducible representation of highest weight in $\Sym^{a-b}(W_{1,0}) \otimes \Sym^b(\bigwedge^2 W_{1,0})$ by construction where $W_{1,0}$ is the $4$-dimensional standard representation of $\Sp(4, \Q_\ell)$. As explained in \cite[Section 1]{faber-van-der-geer-1}, we can lift $W_{a,b}$ to a representation of $\GSp(4, \Q_\ell)$ of dominant weight $aL_1 + bL_2 - (a+b)\eta$ where $\eta$ is the \emph{multiplier representation}, which we denote by $V_{a,b}$. The multiplier representation $\eta$ is defined as $\eta: \GSp(4, \Q_\ell)\to \Q_\ell^\times$ where for any $M \in \GSp(4, \Q_\ell)$ written as 
\[
  M = \begin{pmatrix} A & B \\ C & D \end{pmatrix}
\]
with $A, B, C, D\in \GL(2)$, $\eta(M)$ satisfies
\[
  AD^T - BC^T = \eta(M) I_2.
\]
In particular, this makes $V_{1,0}$ the contragredient representation of the standard representation $V$ of $\GSp(4, \Q_\ell)$, i.e. $V_{1,0} \cong V \otimes \eta^{-1}$.

Let $\pi: \mathcal X_2 \to \mathcal A_2$. By the proper base change theorem (\cite[Corollary VI.2.5]{milne-etale-book}), the stalk of $R^1\pi_*\Q_\ell$ at $[A]\in \mathcal A_2$ is isomorphic to $H^1(A; \Q_\ell)$. Define $\mathbf V_{1,0}$ to be the local system $R^1\pi_*\Q_\ell$. The underlying $\GSp(4, \Q_\ell)$-representation of each stalk of $\mathbf V_{1,0}$ is $V_{1,0}$ and $\mathbf V_{1,0}$ is a local system equipped with a symplectic pairing
\[
  \mathbf V_{1,0} \wedge \mathbf V_{1,0} \to \Q_\ell(-1).
\]
Applying Weyl's construction to the local system $\mathbf V_{1,0}$ yields local systems $\mathbf V_{a,b}$ for all $a \geq b \geq 0$. Each $\mathbf V_{a,b}$ is a summand in $\mathbf V_{1,0}^{\otimes(a+b)}$, so $\mathbf V_{a,b}$ has Hodge weight $a+b$. The underlying $\Sp(4, \Q_\ell)$-representation of $\mathbf V_{a,b}$ is $W_{a,b}$ and the underlying $\GSp(4, \Q_\ell)$-representation of $\mathbf V_{a,b}$ is $V_{a,b}$. 

For all $n \in \Z$, let $\mathbf V_{a,b}(n) := \mathbf V_{a,b} \otimes \Q_\ell(n)$ be the $n$th Tate twist of $\mathbf V_{a,b}$. Tate twists also correspond to tensoring the local systems with the multiplier representation $\eta$, i.e. $\mathbf V_{a,b}(n) = \mathbf V_{a,b} \otimes \eta^{n}$.

\subsection{Cohomological tools.}\label{sec:cohomology}
In this subsection, we list the tools we will need in subsequent sections regarding cohomology computations. First, we set some notation used for the remainder of the paper. We will always denote by $A$ an abelian surface. By $H^*(\mathcal A_2; H^q(X; \Q_\ell))$ for some morphism $f: \mathcal Y \to \mathcal A_2$ with a fiber $X$ over some point in $\mathcal A_2$ where $R^qf_*\Q_\ell$ is locally constant, we will always mean the cohomology of $R^qf_*\Q_\ell$ (see the proof of Proposition \ref{prop:leray}). The prime $\ell$ will always be taken to be coprime to $q$ when working with the base change $\mathcal X_{\overline\F_q}$.

\begin{rmk}\label{rmk:etale-Q-Fq}
Let $\mathcal X = \mathcal A_2$ or $\mathcal X_2^n$ and let $\mathbf V$ be an $\ell$-adic local system on $\mathcal X$. Since $\mathcal X$ is a complement of a normal crossing divisor of a smooth, proper stack over $\Z$ (\cite[Chapter VI]{faltings--chai}), $H^*_{\text{\'et}}(\mathcal X_{\overline \Q}; \mathbf V)$ is unramified at every prime $p \neq \ell$ (\cite[p. 11]{petersen}), i.e. the action of $\Frob_p$ is well-defined. There is an isomorphism 
\[
  H^*_{\text{\'et}}(\mathcal X_{\overline \Q}; \mathbf V) \cong H^*_{\text{\'et}}(\mathcal X_{\overline \F_p}; \mathbf V)
\]
such that the action of $\Gal(\overline \Q_p/\Q_p) \subseteq \Gal(\overline \Q/\Q)$ on the left side factors through the surjection $\Gal(\overline \Q_p/\Q_p) \to \Gal(\overline \F_p/\F_p)$, where $\Gal(\overline \F_p/\F_p)$ acts on the right side. 

To obtain the analogous isomorphism for $\mathcal X_2^{\Sym(n)}$, we recall that $\mathcal X_2^{\Sym(n)}$ is a quotient stack of a scheme $\mathcal X_{2,N}^n$ of some finite group $G$. We now study the Hochschild--Serre spectral sequence for the quotient $\mathcal X_{2,N}^n \to \mathcal X_2^{\Sym(n)}$, where $\mathcal X_{2,N}^n$ is a quasiprojective scheme which is a complement of a normal crossing divisor of a smooth, proper scheme over $\Z$ (\cite[Chapter VI]{faltings--chai}). These spectral sequences are given by
\begin{align*}
  E_2^{p,q} &= H^p(G; H^q((\mathcal X_{2,N}^n)_{\overline \Q}; \Q_\ell)) \implies H^{p+q}((\mathcal X_2^{\Sym(n)})_{\overline \Q}; \Q_\ell), \\
   E_2^{p,q} &= H^p(G; H^q((\mathcal X_{2,N}^n)_{\overline\F_p}; \Q_\ell)) \implies H^{p+q}((\mathcal X_2^{\Sym(n)})_{\overline\F_p}; \Q_\ell).
\end{align*}
Therefore up to semi-simplification, there is the analogous isomorphism
\[
  H^*_{\text{\'et}}((\mathcal X_2^{\Sym(n)})_{\overline \Q}; \Q_\ell) \cong H^*_{\text{\'et}}(\mathcal X_{2,N}^n)_{\overline \Q}; \Q_\ell)^G \cong H^*_{\text{\'et}}(\mathcal X_{2,N}^n)_{\overline \F_p}; \Q_\ell)^G \cong H^*_{\text{\'et}}((\mathcal X_2^{\Sym(n)})_{\overline \F_p}; \Q_\ell).
\]

As stated in the Introduction, we write $H^*(\mathcal X; \mathbf V)$ to denote both $H^*_{\textup{\'et}} (\mathcal X_{\overline \Q}; \mathbf V)$ and $H^*_{\textup{\'et}} (\mathcal X_{\overline\F_q}; \mathbf V)$ for any $\ell$-adic local system $\mathbf V$ with $\ell$ coprime to $q$. As such, all $\ell$-adic local systems $\mathbf V$ in this paper are local systems on $(\mathcal A_2)_{\overline \Q}$ or $(\mathcal A_2)_{\overline \F_q}$.
\end{rmk}

The next two statements are well-known for all $g \geq 1$ but we specialize to the case $g = 2$. Both theorems (for general $g \geq 1$) can be found in \cite{hulek--tommasi}. 
\begin{thm}[Poincar\'e Duality for $\mathcal A_2$]\label{thm:poincare-duality}
For any $a \geq b \geq 0$, 
\[
  H^k_c(\mathcal A_2; \mathbf V_{a, b}) \cong H^{6 - k}(\mathcal A_2; \mathbf V_{a, b})^* \otimes \Q_\ell\left(-3 - a - b\right).
\]
(Recall that $\dim \mathcal A_2 = 3$.)
\end{thm}

\begin{thm}[Deligne's weight bounds]\label{thm:weights}
The mixed Hodge structures on the groups $H^k(\mathcal A_2; \mathbf V_{a, b})$ have weights larger than or equal to $k + a + b$. 
\end{thm}

In Sections \ref{sec:universal-surface}, \ref{sec:fiber-powers}, and \ref{sec:symmetric-powers}, we compute the \'etale cohomology $H^*(\mathcal X; \Q_\ell)$, with $\mathcal X = \mathcal X_2$, $\mathcal X_2^n$, and $\mathcal X_2^{\Sym(n)}$ respectively. In all of these cases, there are morphisms $\pi: \mathcal X \to \mathcal A_2$ to which we want to apply the Leray spectral sequence to obtain the desired results.\begin{thm}[\cite{deligne}]
Let $f: X \to Y$ be a smooth projective morphism of complex varieties. Then the Leray spectral sequence for $f$ degenerates on the $E_2$-page.
\end{thm}
For $N \geq 3$, the projection $\pi: \mathcal X_2[N]^n_{\C} \to \mathcal A_2[N]_{\C}$ is a projective morphism of quasi-projective varieties. Combined with a corollary of the proper base change theorem (\cite[Corollary VI.4.3]{milne-etale-book}), this implies the following useful result:
\begin{cor}\label{cor:deligne}
For all $n \geq 1$, $N \geq 3$, the Leray spectral sequence for $\pi: \mathcal X_2[N]^n_{\overline \Q} \to \mathcal A_2[N]_{\overline \Q}$ degenerates on the $E_2$-page.
\end{cor}

Finally, we state the main tool of this paper. The following proposition gives a Leray spectral sequence for each morphism of stacks $\pi: \mathcal X \to \mathcal A_2$ with $\mathcal X = \mathcal X_2^n$ or $\mathcal X_2^{\Sym(n)}$ using the Leray spectral sequence for schemes in \'etale cohomology.
\begin{prop}\label{prop:leray}
Let $\mathcal X = \mathcal X_2^n$ (resp. $\mathcal X_2^{\Sym(n)}$) and $\pi: \mathcal X \to \mathcal A_2$. There is a spectral sequence
\[
  E_2^{p, q} = H^p(\mathcal A_2; H^q(Z_n; \Q_\ell)) \implies H^{p+q}(\mathcal X; \Q_\ell)
\]
with $Z_n = A^n$ (resp. $Z_n = \Sym^nA$), which degenerates on the $E_2$-page. 
\end{prop}
\begin{proof}
Let $N \geq 3$ and let $\mathcal X_2[N]^n$ be the $n$th fiber power of $\mathcal X_2[N]$ over $\mathcal A_2[N]$ with respect to the projection map $\mathcal X_2[N] \to \mathcal A_2[N]$. Then $\mathcal A_2[N]$ and $\mathcal X_2[N]^n$ are quasi-projective schemes. By the standard Leray spectral sequence for \'etale cohomology (\cite[Theorem 12.7]{milne-etale}) with $\pi_N: \mathcal X_2[N]^n_{\overline \Q} \to \mathcal A_2[N]_{\overline \Q}$,
\[
  H^p(\mathcal A_2[N]_{\overline \Q}; R^q(\pi_N)_*\Q_\ell) \implies H^{p+q}(\mathcal X_2[N]^n_{\overline \Q}; \Q_\ell)
\]
and this spectral sequence degenerates on the $E_2$-page by Corollary \ref{cor:deligne}. Applying a corollary of the proper base change theorem (\cite[Corollary VI.2.5]{milne-etale-book}) with the torsion (constant) sheaf $\Z/\ell^m\Z$, taking inverse limits, and tensoring with $\Q_\ell$,
\[
  \bigoplus_{p+q = k} H^p(\mathcal A_2[N]_{\overline \Q}; H^q(A^n; \Q_\ell))\cong H^k(\mathcal X_2[N]^n_{\overline \Q}; \Q_\ell) 
\]
as Galois representations up to semi-simplification. Then by the Hochschild--Serre spectral sequence for the $\Sp(4, \Z/N\Z)$-quotient $\mathcal X_2[N]^n_{\overline \Q} \to (\mathcal X_2^n)_{\overline \Q}$, where $\Sp(4, \Z/N\Z)$ acts diagonally on $\mathcal X_2^n[N]$, and the $\Sp(4, \Z/N\Z)$-quotient $\mathcal A_2[N]_{\overline \Q} \to (\mathcal A_2)_{\overline \Q}$,
\[
\bigoplus_{p+q = k} H^p(\mathcal A_2[N]_{\overline \Q}; H^q(A^n; \Q_\ell))^{\Sp(4, \Z/N\Z)} \cong H^{k}(\mathcal X_2[N]^n_{\overline \Q}; \Q_\ell)^{\Sp(4, \Z/N\Z)} \cong H^{k}((\mathcal X_2^n)_{\overline \Q}; \Q_\ell)
\]
and
\[
  H^p(\mathcal A_2[N]_{\overline \Q}; H^q(A^n; \Q_\ell))^{\Sp(4, \Z/N\Z)} \cong H^p((\mathcal A_2)_{\overline \Q}; H^q(A^n; \Q_\ell)),
\]
where on the left, $H^q(A^n)$ is the local system corresponding to the respective $\Sp(4, \Z)[N]$-representation, while on the right, $H^q(A^n)$ is the local system corresponding to the respective $\Sp(4, \Z)$-representation. Therefore, taking $\Sp(4, \Z/N\Z)$-invariants in the spectral sequence for $\pi_N$, which one can do by naturality of that sequence, gives the following $E_2$-page of a spectral sequence
\[
  E_2^{p,q} = H^p((\mathcal A_2)_{\overline \Q}; H^q(A^n; \Q_\ell)) \implies H^{p+q}((\mathcal X_2^n)_{\overline \Q}; \Q_\ell)
\]
degenerating on the $E_2$-page. By Remark \ref{rmk:etale-Q-Fq}, the spectral sequence for $(\mathcal X_2^n)_{\overline \F_q} \to (\mathcal A_2)_{\overline \F_q}$ must also degenerate on the $E_2$-page. Therefore, we have now proven the Proposition for $\mathcal X = (\mathcal X_2^n)_{\overline \Q}$ and $(\mathcal X_2^n)_{\overline\F_q}$. Lastly, again by the Hochschild--Serre spectral sequence,
\[
  H^{k}(\mathcal X_2^{\Sym(n)}; \Q_\ell) \cong H^k(\mathcal X_2^n; \Q_\ell)^{S_n}.
\]
Because $S_n$ acts trivially on $\mathcal A_2$, 
\[
  H^p((\mathcal A_2)_{\overline \Q}; H^q(A^n))^{S_n} \cong H^p((\mathcal A_2)_{\overline \Q}; H^q(A^n)^{S_n}) \cong H^p((\mathcal A_2)_{\overline \Q}; H^q(\Sym^n A))
\]
and
\[
  H^p((\mathcal A_2)_{\overline \F_q}; H^q(A^n))^{S_n} \cong H^p((\mathcal A_2)_{\overline \F_q}; H^q(A^n)^{S_n}) \cong H^p((\mathcal A_2)_{\overline \F_q}; H^q(\Sym^n A))
\]
where $H^q(\Sym^n A)$ is again an $\Sp(4, \Z)$-representation. Again by naturality, taking $S_n$-invariants in the spectral sequence for $(\mathcal X_2^n)_{\overline \Q} \to (\mathcal A_2)_{\overline \Q}$ and $(\mathcal X_2^n)_{\overline \F_q} \to (\mathcal A_2)_{\overline \F_q}$ gives
\begin{align*}
  E_2^{p,q} = H^p((\mathcal A_2)_{\overline \Q}; H^q(\Sym^n A)) \implies H^{p+q}((\mathcal X_2^{\Sym(n)})_{\overline \Q}; \Q_\ell), \\
  E_2^{p,q} = H^p((\mathcal A_2)_{\overline \F_q}; H^q(\Sym^n A)) \implies H^{p+q}((\mathcal X_2^{\Sym(n)})_{\overline \F_q}; \Q_\ell)
\end{align*}
which both degenerate on the $E_2$-page.
\end{proof}

\section{Cohomology of the Universal Abelian Surface}\label{sec:universal-surface}
In this section, we study the cohomology of $\mathcal X_2$ using $\pi: \mathcal X_2 \to \mathcal A_2$. We first need to compute the following local systems.

\begin{lem}\label{lem:X2-local-systems}
There are isomorphisms of local systems
\[
  H^k(A; \Q_\ell) \cong \begin{cases}
  \Q_\ell & k = 0\\
  \mathbf V_{1, 0} & k = 1 \\
  \Q_\ell(-1) \oplus \mathbf V_{1, 1} & k = 2 \\
  \mathbf V_{1,0}(-1) & k = 3 \\
  \Q_\ell(-2) & k = 4 \\
  0 & k > 4.
  \end{cases}
\]
\end{lem}
\begin{proof}
By \cite[p. 238]{faltings--chai}, smooth $\ell$-adic sheaves on $\mathcal A_2$ correspond to continuous representations of the arithmetic fundamental group $\pi_1(\mathcal A_2)$ of $\mathcal A_2$ after choosing a base point; we can view $\GSp(4)$ as the arithmetic fundamental group of $\mathcal A_2$ after a choice of base point (cf. \cite[p. 6]{v-rank-1-eisenstein}). For any abelian surface $A$, there is an isomorphism $\bigwedge^k H^1(A; \Q_\ell) \to H^k(A; \Q_\ell)$ for all $k \geq 0$ given by the cup-product pairing by \cite[Theorem 12.1]{milne-abvar}. Therefore, there is an isomorphism of local systems between $H^k(A; \Q_\ell)$ and $\bigwedge^k \mathbf V_{1, 0}$, the local system corresponding to the $\GSp(4, \Q_\ell)$-representation $\bigwedge^k V_{1,0}$. We deccompose the local system $\bigwedge^k \mathbf V_{1,0}$ into a direct sum of local systems of the form $\mathbf V_{a,b}(n)$ with $a \geq b \geq 0$ and $n \in \Z$ corresponding to irreducible $\GSp(4, \Q_\ell)$-representations.

We first consider the decomposition of the $\GSp(4, \Q_\ell)$-representation $\bigwedge^k V_{1,0}$ into irreducible $\GSp(4, \Q_\ell)$-representations, where $V_{a,b}$ denotes the irreducible $\GSp(4, \Q_\ell)$-representation corresponding to the partition $a\geq b \geq 0$ as explained in Section \ref{sec:local-systems}. Recall also that $W_{a,b}$ denotes the irreducible $\Sp(4, \Q_\ell)$-representation corresponding to the partition $a \geq b \geq 0$. The decomposition of $\bigwedge^k W_{1,0}$ as $\Sp(4, \Q_\ell)$-representation is
\[
  \bigwedge^k W_{1,0} \cong \begin{cases}
  W_{0,0} & k = 0, 4, \\
  W_{1,0} & k = 1, 3, \\
  W_{0,0} \oplus W_{1,1} & k = 2
  \end{cases}
\]
as given in \cite[Chapter 16]{fulton--harris}. 

The $\Sp(4, \Q_\ell)$-representation decomposition above determines the corresponding $\GSp(4, \Q_\ell)$-representation $\bigwedge^k V_{1,0}$ up to tensoring by the multiplier representation $\eta$ (see Section \ref{sec:local-systems}). In particular, this means that if $W_{a,b} \subseteq \bigwedge^k W_{1,0}$ for some $a \geq b \geq 0$ then there exists some $q \in \Z$ such that $V_{a,b} \otimes \eta^{q} \subseteq \bigwedge^k V_{1,0}$. To determine the $q$ for each summand $V_{a,b} \otimes \eta^q \subseteq \bigwedge^k V_{1,0}$, it will suffice to consider the action of all scalar matrices $m I_4 \in \GSp(4, \Q_\ell)$ on $\bigwedge^k V_{1,0}$. 

Recall that $V_{1,0}$ is the contragredient of the standard representation as a $\GSp(4, \Q_\ell)$-representation $V$, i.e. $V_{1,0} \cong V \otimes \eta^{-1}$. Apply the definition of the multiplier representation from Section \ref{sec:local-systems} to see that $\eta(mI_4) = m^2$. For any $v \in V_{a,b} \subseteq V_{1,0}^{\otimes(a+b)} = V^{\otimes(a+b)} \otimes \eta^{-(a+b)}$, 
\[
  m I_4 \cdot v = (m^{(a+b)} v) \cdot m^{-2(a+b)} = m^{-(a+b)} v
\]
and so $mI_4 \cdot w = m^{-(a+b) +2q}w$ for any $w \in V_{a,b} \otimes \eta^q \subseteq V_{1,0}^{\otimes (a+b)}\otimes \eta^q$. On the other hand, for any $w \in V_{a,b} \otimes \eta^q \subseteq \bigwedge^k V_{1,0} \subseteq V_{1,0}^{\otimes k}$,
\[
  m I_4 \cdot w = (m^kw) \cdot m^{-2k} = m^{-k}w.
\]
This shows that $m^{-(a+b) + 2q}w = m^{-k}w$, which implies that $-k = -(a+b)+2q$ and so $q = \frac{(a+b)-k}{2}$. Therefore as $\GSp(4, \Q_\ell)$-representations,
\[
  \bigwedge^k V_{1,0} \cong \begin{cases}
  V_{0,0} & k = 0,\\
  V_{1,0} & k = 1,\\
  V_{0,0}\otimes \eta^{-1} \oplus V_{1,1} & k = 2, \\
  V_{1,0} \otimes \eta^{-1} & k = 3, \\
  V_{0,0}\otimes \eta^{-2} & k = 4, \\
  0 & k > 4.
  \end{cases}
\]
Tensoring by $\eta(q)$ corresponds to tensoring with $\Q_\ell(q)$ for all $n \in \Z$ (see Section \ref{sec:local-systems}). Rewriting the above decomposition of $\bigwedge^k V_{1,0}$ in terms of local systems proves this lemma.
\end{proof}

Our main tool is \cite[Theorem 2.1]{petersen}, restated below for convenience. Before we do so, we need to establish some notation, which agrees with that of \cite{petersen}.

Let $s_k$ be the dimension of the space of cusp forms of $\SL(2, \Z)$ of weight $k$. For $j \geq 0$ and $k \geq 3$, let $s_{j,k}$ be the dimension of the space of vector-valued Siegel cusp forms for $\Sp(4, \Z)$ transforming according to the representation $\Sym^j \otimes \det^k$. Let $\rho_f$ be the $2$-dimensional $\ell$-adic Galois representation of weight $k-1$ of the normalized cusp eigenform $f$ for $\SL(2, \Z)$, as given by \cite{deligne-1969}, and let $S_k = \bigoplus_f \rho_f$ be the direct sum of such representations for $k$. Let $\tau_f$ be the $4$-dimensional $\ell$-adic Galois representation of the vector-valued Siegel cusp eigenform $f$ of type $\Sym^j \otimes \det^k$ as given by \cite{weissauer}, and let $S_{j,k} = \bigoplus_f \tau_f$. Let $s_k'$ be the number of normalized cusp eigenforms of weight $k$ for $\SL(2, \Z)$ for which $L(f, \frac 12)$ vanishes. 

Finally, we need to describe the Galois representations $\overline S_{j,k}$. These representations satisfy the condition that $\overline S_{j,k} = S_{j,k}$ if $j \neq 0$ or $k \equiv 1 \pmod 2$. Otherwise, $\overline S_{j,k}$ is a subrepresentation of $S_{j,k}$ that can be determined in a prescribed way. Because the only property of $\overline S_{j,k}$ we will use in this paper is that it is a subrepresentation of $S_{j,k}$, we refer the reader to \cite[p. 3]{petersen} for the specific definition.

Because every abelian surface has an involution which acts by multiplication by $(-1)^k$ on each stalk of $\mathbf V_{1,0}^{\otimes k}$, and each $\mathbf V_{a, b}$ is a summand of $\mathbf V_{1,0}^{\otimes(a+b)}$, the cohomology $H^p(\mathcal A_2; \mathbf V_{a, b})$ vanishes if $a+b$ is odd. (See \cite[\S 1]{faber-van-der-geer-1} or \cite[p. 3]{petersen}.)
\begin{thm}[Petersen, {\cite[Theorem 2.1]{petersen}}]\label{thm:petersen}
Suppose $(a, b) \neq (0, 0)$, and that $a + b$ is even. Then

\begin{enumerate}
  \item In degrees $k \neq 2$, $3$, $4$, 
  \[
    H_c^k(\mathcal A_2; \mathbf V_{a, b}) = 0. 
  \]

  \item In degree $4$,
  \[
    H_c^4(\mathcal A_2; \mathbf V_{a, b}) = \begin{cases}
    s_{a+b+4}\Q_\ell(-b-2) & a = b\text{ even} \\
    0 & \text{otherwise}.
    \end{cases}
  \]

  \item In degree $3$, up to semi-simplification,
  \begin{align*}
    H^3_c(\mathcal A_2; \mathbf V_{a, b}) = &\overline S_{a - b, b+3} \oplus s_{a+b+4} S_{a-b+2}(-b-1) \oplus S_{a+3} \\
    &\oplus \begin{cases}
    s'_{a+b+4}\Q_\ell(-b-1) & a = b\text{ even}\\
    s_{a+b+4}\Q_\ell(-b-1) & \text{otherwise}
    \end{cases} \\
    &\oplus \begin{cases}
    \Q_\ell & a = b \text{ odd}\\
    0 & \text{otherwise}
    \end{cases} \\
    &\oplus \begin{cases}
    \Q_\ell(-1) & b = 0\\
    0 & \text{otherwise}.
    \end{cases}
  \end{align*}

  \item In degree $2$, up to semi-simplification,
  \begin{align*}
  H_c^2(\mathcal A_2; \mathbf V_{a, b}) = &S_{b+2} \oplus s_{a-b+2}\Q_\ell \\
  &\oplus \begin{cases}
  s'_{a+b+4}\Q_\ell(-b-1) & a = b\text{ even}\\
  0 & \text{otherwise}
  \end{cases} \\
  &\oplus \begin{cases}
  \Q_\ell & a > b > 0 \text{ and } a, b \text{ even}\\
  0 & \text{otherwise}.
  \end{cases}
  \end{align*}
\end{enumerate}
\end{thm}
\begin{rmk}
Although we do not need the full power of Theorem \ref{thm:petersen} to compute $H^*(\mathcal X_2; \Q_\ell)$, the existence of such a result is important for the calculations in Sections \ref{sec:fiber-powers} and \ref{sec:symmetric-powers}.
\end{rmk}

We now give examples of computations using Theorem \ref{thm:petersen}. The next corollary gives all applications of this theorem that we explicitly use in the rest of the paper. All results are up to semi-simplification.
\begin{cor}\label{cor:petersen}
For all integers $k \geq 0$, 
\begin{align}
  H^1(\mathcal A_2; \mathbf V_{a,b}) &= 0 \quad\quad \text{for all }a \geq b \geq 0, \label{eqn:cor-petersen-deg1}\\
  H^k(\mathcal A_2; \mathbf V_{1, 1}) &= \begin{cases}
  \Q_\ell(-5) & k = 3 \\
  0 & \text{otherwise,}
  \end{cases} \label{eqn:cor-petersen-1,1} \\
  H^k(\mathcal A_2; \mathbf V_{2,0}) &= \begin{cases}
  \Q_\ell(-4) & k = 3 \\
  0 & \text{otherwise,}
  \end{cases} \label{eqn:cor-petersen-2,0}\\
  H^k(\mathcal A_2; \mathbf V_{2,2}) &= 0 \quad\quad \text{for all }k \geq 0. \label{eqn:cor-petersen-2,2}
\end{align}
\end{cor}
\begin{proof}
The smallest weight possible for nonzero cusp forms of $\SL(2, \Z)$ is $12$. (For example, see \cite[Theorem 7.4]{serre-arithmetic}.) Thus $s_k' = s_k = 0$ and $S_k = 0$ for all $k < 12$. 

By Theorem \ref{thm:petersen},
\begin{align*}
H_c^5(\mathcal A_2; \mathbf V_{a,b}) &= 0 \quad\quad \text{for all }a \geq b \geq 0, \\
H_c^4(\mathcal A_2; \mathbf V_{1, 1}) &= 0, \\
H_c^3(\mathcal A_2; \mathbf V_{1, 1}) &= \overline S_{0, 4} \oplus s_6S_2(-2) \oplus S_4 \oplus s_6\Q_\ell(-2) \oplus \Q_\ell = \overline S_{0,4} \oplus \Q_\ell, \\
H_c^2(\mathcal A_2; \mathbf V_{1, 1}) &= S_3 \oplus s_2\Q_\ell = 0, \\
H_c^4(\mathcal A_2; \mathbf V_{2,0}) &= 0, \\
H_c^3(\mathcal A_2; \mathbf V_{2,0}) &= \overline S_{2, 3} \oplus s_6 S_4(-1) \oplus S_5 \oplus s_6\Q_\ell(-1) \oplus \Q_\ell(-1) = \overline S_{2,3} \oplus \Q_\ell(-1), \\
H_c^2(\mathcal A_2; \mathbf V_{2,0}) &= S_2 \oplus s_4\Q_\ell = 0, \\
H_c^4(\mathcal A_2; \mathbf V_{2,2}) &= s_8\Q_\ell(-4) = 0, \\
H_c^3(\mathcal A_2; \mathbf V_{2,2}) &= \overline S_{0, 5} \oplus s_8 S_2(-3) \oplus S_5 \oplus s'_{8}\Q_\ell(-3) = \overline S_{0,5}, \\
H_c^2(\mathcal A_2; \mathbf V_{2,2}) &= S_4 \oplus s_2\Q_\ell \oplus s'_8 \Q_\ell(-3) = 0.
\end{align*}
By computations in \cite[p. 249]{wakatsuki}, 
\begin{align*}
\overline S_{0, 4} &\subseteq S_{0, 4} = 0, \\
\overline S_{0, 5} &\subseteq S_{0 ,5} = 0.
\end{align*} 
By \cite[Lemma 2.1]{ibukiyama}, $S_{j,k} = 0$ for all $0 \leq k \leq 4$ and $j \leq 14$, and so 
\[
  \overline S_{2,3} \subseteq S_{2,3} = 0.
\]
Finally, by Poincar\'e Duality (Theorem \ref{thm:poincare-duality}), 
\begin{align*}
H^1(\mathcal A_2; \mathbf V_{a,b})^* &\cong H_c^5(\mathcal A_2; \mathbf V_{a,b}) \otimes \Q_\ell(3+a+b) = 0 \quad\quad \text{for all } a \geq b \geq 0,\\
H^{k}(\mathcal A_2; \mathbf V_{1,1})^* &\cong H_c^{6-k}(\mathcal A_2; \mathbf V_{1,1}) \otimes \Q_\ell(5) = \begin{cases}
\Q_\ell(5) & k = 3 \\
0 & \text{otherwise,}
\end{cases}\\
H^{k}(\mathcal A_2; \mathbf V_{2, 0})^* &\cong H_c^{6-k}(\mathcal A_2; \mathbf V_{2,0}) \otimes \Q_\ell(5) = \begin{cases}
\Q_\ell(4) & k = 3 \\
0 & \text{otherwise,} 
\end{cases}\\
H^{k}(\mathcal A_2; \mathbf V_{2,2})^* &\cong H_c^{6-k}(\mathcal A_2; \mathbf V_{2,2}) \otimes \Q_\ell(7) = 0 \quad\quad \text{for all }k \geq 0. \qedhere
\end{align*}
\end{proof}

With the above preliminaries in hand, we can now prove our first main result, Theorem \ref{thm:universal-surface}. We restate the theorem here for convenience.
\begin{thm:universal-surface}
The cohomology of the universal abelian surface $\mathcal X_2$ is given by
\[
  H^k(\mathcal X_2; \Q_\ell) = \begin{cases}
  \Q_\ell & k = 0 \\
  0 & k = 1, 3, k \geq 7\\
  2\Q_\ell(-1) & k = 2 \\
  2\Q_\ell(-2) & k = 4 \\
  \Q_\ell(-5) & k = 5 \\
  \Q_\ell(-3) & k = 6 \\
  \end{cases}
\]
up to semi-simplification. 
\end{thm:universal-surface}
\begin{proof}
By Proposition \ref{prop:leray}, there is a spectral sequence
\[
  E_2^{p, q} = H^p(\mathcal A_2; H^q(A; \Q_\ell)) \implies H^{p+q}(\mathcal X_2; \Q_\ell)
\]
which degenerates on the $E_2$-page. Combining Lemma \ref{lem:X2-local-systems}, Theorem \ref{thm:a2}, and Corollary \ref{cor:petersen} shows that the $E_2$-page is as in Figure \ref{fig:X2-spectral-sequence}. The theorem now follows directly, since up to semi-simplification,
\[
  H^k(\mathcal X_2; \Q_\ell) = \bigoplus_{p+q = k} E_2^{p, q}. \qedhere
\]
\end{proof}

\begin{figure}
\centering
\includegraphics{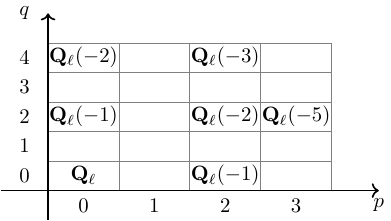}
\caption{The nonzero terms on the $E_2$-page of the Leray spectral sequence for $\pi: \mathcal X_2 \to \mathcal A_2$.}
\label{fig:X2-spectral-sequence}
\end{figure}

\section{Cohomology of Fiber Powers of the Universal Abelian Surface}\label{sec:fiber-powers}
In this section, we compute $H^*(\mathcal X_2^n; \Q_\ell)$. In particular, we give a procedure for computing this for general $n$, and then give the specific results that follow for the case $n = 2$. For brevity, we omit the coefficients when it is clear that we mean the constant ones and write $H^*(X)$ to mean $H^*(X; \Q_\ell)$. 
\subsection{Computations for general $n$.}\label{sec:fiber-powers-n}
Let $\pi: \mathcal X_2 \to \mathcal A_2$ and let $\pi^n: \mathcal X_2^n \to \mathcal A_2$. We first need to consider the following local systems.
\begin{lem}\label{lem:X2-n-local-systems}
There are isomorphisms of local systems
\begin{align*}
  H^k(A^n) &\cong \bigoplus_{\sum_{i=1}^n k_i = k} \left(\bigotimes_i H^{k_i}(A)\right)\\
  &\cong H^k(A^{n-1}) \oplus \left(\mathbf V_{1,0} \otimes H^{k-1}(A^{n-1}) \right) \oplus \left((\Q_\ell(-1) \oplus \mathbf V_{1,1}) \otimes H^{k-2}(A^{n-1})\right) \\
  &\quad \oplus \left(\mathbf V_{1,0}(-1) \otimes H^{k-3}(A^{n-1}) \right) \oplus H^{k-4}(A^{n-1})(-2)
\end{align*}
where we say $H^k(A^{n-1}) = 0$ if $k < 0$. For appropriate constants $m_{a,b}(H^k(A^n))$,
\begin{equation}\label{eqn:tate-twists}
  H^k(A^n) \cong \bigoplus_{a \geq b \geq 0} m_{a,b}(H^k(A^n))\mathbf V_{a,b}\left(\frac{a+b- k}{2}\right).
\end{equation}
\end{lem}
\begin{proof}
The K\"unneth isomorphism applied to the local system $H^k(A^n)$ says
\begin{align*}
  H^k(A^n) &\cong \bigoplus_{k_1 + k_2 = k} H^{k_1}(A^{n-1}) \otimes H^{k_2}(A) \\
  &\cong \bigoplus_{k_2 = 0}^4 H^{k_2}(A) \otimes H^{k-k_2}(A^{n-1}),
\end{align*}
from which the first half of the lemma follows from Lemma \ref{lem:X2-local-systems}. 

For the third isomorphism, the same proof as that of Lemma \ref{lem:X2-local-systems} applies as follows. As explained in Section \ref{sec:local-systems}, $W_{a,b}$ denotes the irreducible $\Sp(4, \Q_\ell)$-representation corresponding to the partition $a \geq b \geq 0$ and $V_{a,b}$ denotes an irreducible $\GSp(4, \Q_\ell)$-representation whose underlying $\Sp(4, \Q_\ell)$-representation is $W_{a,b}$. For each summand in the first isomorphism, suppose as $\Sp(4, \Q_\ell)$-representations
\[
  W_{a,b} \subseteq \bigotimes_i H^{k_i}(A) \cong \bigotimes_i \left(\bigwedge^{k_i} H^1(A)\right)
\]
which implies that for some $q \in \Z$, $V_{a,b} \otimes \eta^q \subseteq \bigotimes_i \left(\bigwedge^{k_i} V_{1,0}\right) \subseteq V_{1,0}^{\otimes k}$ as $\GSp(4, \Q_\ell)$-representations. Exactly the same computation as in the proof of Lemma \ref{lem:X2-local-systems} by applying scalar matrices $m I_4 \in \GSp(4, \Q_\ell)$ to both $V_{1,0}^{\otimes k}$ and $V_{a,b}\otimes \eta^q$ shows that $q = \frac{a+b-k}{2}$. Finally, note that tensoring with $\eta$ corresponds to Tate twists, which proves the last isomorphism.
\end{proof}

In order to apply Theorem \ref{thm:petersen} like in Section \ref{sec:universal-surface}, we need to decompose each $H^k(A^n; \Q_\ell)$ into irreducible representations of $\Sp(4, \Q_\ell)$. Consider the restriction of the irreducible $\GSp(4, \Q_\ell)$-representation $V_{a,b}$ (corresponding to the partition $a \geq b \geq 0$) to $\Sp(4,\Q_\ell)$; as in Section \ref{sec:local-systems}, we denote this irreducible $\Sp(4, \Q_\ell)$-representation by $W_{a,b}$. We account for the $\GSp(4, \Q_\ell)$-representation structures at the end by using Lemma \ref{lem:X2-n-local-systems}(\ref{eqn:tate-twists}). We set $W_{a,b} := 0$ if $a < b$ or $b < 0$. Lemma \ref{lem:X2-n-local-systems} gives rise to a recursive computation (in $n$) for this decomposition to which we will apply the following two lemmas:
\begin{lem}\label{lem:rep-thy-1,0}
For $a \geq b \geq 0$, 
\[
  W_{1, 0} \otimes W_{a, b} = W_{a-1, b} \oplus W_{a+1, b} \oplus W_{a, b-1} \oplus W_{a, b+1}. 
\]
\end{lem}
\begin{lem}\label{lem:rep-thy-1,1}
If $a > b \geq 0$, then
\[
  W_{1, 1} \otimes W_{a, b} = W_{a-1, b-1} \oplus W_{a-1, b+1} \oplus W_{a, b} \oplus W_{a+1, b-1}\oplus W_{a+1, b+1}.
\]
If $a \geq 0$, then
\[
  W_{1, 1} \otimes W_{a, a} = W_{a-1, a-1} \oplus W_{a+1, a-1}\oplus W_{a+1, a+1}.
\]
\end{lem}
The proofs of Lemmas \ref{lem:rep-thy-1,0} and \ref{lem:rep-thy-1,1} apply combinatorial theorems to decompose tensor products of irreducible $\Sp(4, \Q_\ell)$-representations into a direct sum of irreducible $\Sp(4, \Q_\ell)$-representations. We summarize the necessary combinatorial results and prove Lemmas \ref{lem:rep-thy-1,0} and \ref{lem:rep-thy-1,1} in Appendix \ref{sec:tensors}; the rest of this paper is independent of the content of Appendix \ref{sec:tensors}.

We are now able to give a recursive formula (in $n$) for the multiplicity of a given $W_{a, b}$ in $H^k(A^n;\Q_\ell)$ as $\Sp(4, \Q_\ell)$-representations. We do so in pieces after establishing some notation.
\begin{defn}
For any $(a, b) \in \N^2$ and any $\Sp(4, \Q_\ell)$-representation $V$, let
\[
  m_{a, b}(V) = \langle W_{a, b}, V\rangle.
\]
Recall that if $a < b$ or $b < 0$, then $W_{a,b} = 0$ so $m_{a, b}(V) = 0$ for all $V$.
\end{defn} 
\begin{rmk}
Lemma \ref{lem:X2-local-systems} determines $m_{a, b}(H^q(A))$ for all $a \geq b \geq 0$ and all $q \geq 0$. 
\end{rmk}
\begin{defn}\label{defn:indicator}
For any set $S$, denote the indicator function by $\ind_S(a, b)$ with
\[
  \ind_S(a, b) = \begin{cases}
  1 & (a, b)\in S \\
  0 & (a, b)\notin S.
  \end{cases}
\]
\end{defn}

The following two lemmas establish formulas that are necessary to give a recursive formula for $m_{a, b}(H^k(A^n))$ in $n$. The proofs are completely straight-forward but are included for completeness. 
\begin{lem}\label{lem:rec-13}
For any $a \geq b \geq 0$, $k \geq 0$, and $\Sp(4, \Q_\ell)$-representation $V$,
\begin{align*}
  \langle W_{a, b}, W_{1, 0} \otimes V \rangle = m_{a+1, b}(V) + m_{a, b+1}(V) + m_{a-1, b}(V) + m_{a, b-1}(V).
\end{align*}
\end{lem}
\begin{proof}
By Section \ref{sec:local-systems}, $V$ is a direct sum of irreducible $\Sp(4, \Q_\ell)$-representations, which are all of the form $W_{A,B}$ for some $A \geq B \geq 0$. Let $I(a, b) = \{(a+1, b), (a, b+1), (a-1, b), (a, b-1)\}$. By Lemma \ref{lem:rep-thy-1,0},
\begin{align*}
  \langle W_{a, b}, W_{1, 0} \otimes V \rangle &= \left\langle W_{a, b}, W_{1, 0} \otimes \left(\bigoplus_{(A, B)\in I(a, b)}m_{A, B}(V) W_{A, B} \right)\right\rangle \\
  &= \sum_{(A, B) \in I(a, b)} m_{A, B}(V) \langle W_{a, b}, W_{1, 0} \otimes W_{A, B} \rangle \\
  &= \sum_{(A, B) \in I(a, b)}m_{A, B}(V) \langle W_{a, b}, W_{A+1, B} \oplus W_{A, B+1} \oplus W_{A-1, B} \oplus W_{A, B-1} \rangle \\
  &= m_{a+1, b}(V) + m_{a, b+1}(V) + m_{a-1, b}(V) + m_{a, b-1}(V). \qedhere
\end{align*}
\end{proof}
\begin{lem}\label{lem:rec-2}
For any $a \geq b \geq 0$ and $\Sp(4, \Q_\ell)$-representation $V$,
\begin{align*}
\langle & W_{a, b}, (W_{0,0} \oplus W_{1, 1})\otimes V \rangle \\
&= (2- \ind_{\{a = b\}}(a, b))m_{a, b}(V) + m_{a-1, b-1}(V)+ m_{a-1, b+1}(V) + m_{a+1, b-1}(V) + m_{a+1, b+1}(V). 
\end{align*}
\end{lem}
\begin{proof}
By Section \ref{sec:local-systems}, $V$ is a direct sum of irreducible $\Sp(4, \Q_\ell)$-representations, which are all of the form $W_{A,B}$ for some $A \geq B \geq 0$. For all $a \geq b \geq 0$, $\langle W_{a, b}, W_{0,0} \otimes V \rangle = m_{a,b}(V)$. Now let $J(a, b) = \{(a-1, b-1), (a-1, b+1), (a, b), (a+1, b-1), (a+1, b+1)\}$. By Lemma \ref{lem:rep-thy-1,1},
\begin{align*}
  \langle W_{a, b}, W_{1,1} \otimes V \rangle &= \left\langle W_{a, b}, W_{1,1} \otimes \left(\bigoplus_{(A, B) \in J(a, b)}m_{A, B}(V) W_{A, B}\right) \right\rangle \\
  &= \sum_{(A, B) \in J(a, b)} m_{A, B}(V) \langle W_{a, b}, W_{1,1} \otimes W_{A, B} \rangle.
\end{align*}
For all $(A, B) \in J(a,b)$ with $A > B$, Lemma \ref{lem:rep-thy-1,1} says
\[
  m_{A, B}(V) \langle W_{a, b}, W_{1,1} \otimes W_{A, B} \rangle = m_{A, B}(V) \left\langle W_{a, b}, \bigoplus_{(C, D) \in J(A, B)}W_{C, D} \right\rangle = m_{A, B}(V).
\]
If $A > B$ for all $(A, B) \in J(a,b)$ then $a \neq b$ and 
\begin{align*}
\langle W_{a,b}, (\Q_\ell \oplus W_{1,1})\otimes V \rangle &= \langle W_{a,b}, \Q_\ell \otimes V \rangle + \langle W_{a,b}, W_{1,1} \otimes V \rangle \\
&= m_{a,b}(V) + \sum_{(A,B) \in J(a,b)} m_{A,B}(V) \\
&= 2m_{a,b}(V) + m_{a-1, b-1}(V)+ m_{a-1, b+1}(V) + m_{a+1, b-1}(V) + m_{a+1, b+1}(V),
\end{align*}
which proves the lemma in this case.

If $(A, A) \in J(a,b)$ with $A \geq 0$, then one of the following cases occur:
\begin{enumerate}
  \item $A+1 = a = b$ and 
  \begin{align*}
    m_{A,A}(V) \langle W_{a,b}, W_{1,1} \otimes W_{A,A} \rangle &= m_{A,A}(V) \langle W_{A+1, A+1}, W_{A-1, A-1} \oplus W_{A+1, A-1} \oplus W_{A+1, A+1}\rangle\\
    &= m_{a-1, a-1}(V).
  \end{align*}
  \item $A+1 = a = b+2$ and
  \begin{align*}
    m_{A,A}(V) \langle W_{a,b}, W_{1,1} \otimes W_{A,A} \rangle &= m_{A,A}(V) \langle W_{A+1,A-1}, W_{A-1, A-1} \oplus W_{A+1, A-1} \oplus W_{A+1, A+1}\rangle\\
    &= m_{a-1, b+1}(V).
  \end{align*}
  \item $A = a = b$ and
  \begin{align*}
    m_{A,A}(V) \langle W_{a,b}, W_{1,1} \otimes W_{A,A} \rangle &= m_{A,A}(V) \langle W_{A,A}, W_{A-1, A-1} \oplus W_{A+1, A-1} \oplus W_{A+1, A+1}\rangle \\
    &= 0.
  \end{align*}
  \item $A-1 = a = b-2$, but this implies that $a < b$. 
  \item $A-1 = a = b$ and
  \begin{align*}
    m_{A,A}(V) \langle W_{a,b}, W_{1,1} \otimes W_{A,A} \rangle &= m_{A,A}(V) \langle W_{A-1,A-1}, W_{A-1, A-1} \oplus W_{A+1, A-1} \oplus W_{A+1, A+1}\rangle \\
    &= m_{a+1, a+1}(V).
  \end{align*}
\end{enumerate}
Rearranging, these cases reduce to one of the following:
\begin{enumerate}
  \item $a = b$ and $(a-1, a-1), (a, a), (a+1, a+1) \in J(a,a)$, so
  \begin{align*}
  \langle W_{a,a}, (W_{0,0} \oplus W_{1,1}) \otimes V \rangle &= \langle W_{a,a}, W_{0,0} \otimes V \rangle + \langle W_{a,a}, W_{1,1} \otimes V \rangle \\
  &= m_{a,a}(V) + \sum_{(A,B) \in J(a,a)} m_{A,B}(V) \langle W_{a,a}, W_{1,1} \otimes W_{A,B} \rangle \\
  &= m_{a,a}(V) + \left( \sum_{(A, A) \in J(a,a)} m_{A,A}(V) \langle W_{a,a}, W_{1,1} \otimes W_{A,A} \rangle \right) \\
  &\quad + m_{a+1, a-1}(V) \langle W_{a,a}, W_{1,1} \otimes W_{a+1, a-1} \rangle \\
  &= m_{a,a}(V) + m_{a-1,a-1}(V) + m_{a+1,a+1}(V) + m_{a+1,a-1}(V)
  \end{align*}
  which proves the lemma in this case.
  \item $a = b+2$ and $(a-1, a-1) \in J(a, b)$, so
  \begin{align*}
  \langle W_{a,b}, (W_{0,0} \oplus W_{1,1}) \otimes V \rangle &= \langle W_{a,b}, W_{0,0} \otimes V \rangle + \langle W_{a,b}, W_{1,1} \otimes V \rangle \\
  &= m_{a,b}(V) + \sum_{(A,B) \in J(a,b)} m_{A,B}(V) \langle W_{a,b}, W_{1,1} \otimes W_{A,B} \rangle \\
  &= m_{a,b}(V) + m_{a-1,b+1}(V) \langle W_{a,b}, W_{1,1} \otimes W_{a-1,a-1} \rangle + \sum_{\substack{(A,B) \in J(a,b) \\ (A,B) \neq (a-1,a-1)}} m_{A,B}(V)\\
  &= m_{a,b}(V) +\sum_{(A,B) \in J(a,b)} m_{A,B}(V)
  \end{align*}
  which proves the lemma in this case. \qedhere
\end{enumerate}
\end{proof}

Combining all of the lemmas of this subsection shows that we have determined a recursive formula for all parts of the first identity of the following proposition.
\begin{prop}\label{prop:recursive_mult}
Let $a \geq b \geq 0$, $n \geq 1$, and $k \geq 0$. Viewing $H^*(A^N)$ as $\Sp(4, \Q_\ell)$-representations,
\begin{equation}\label{eqn:ab-for-Hk}
\begin{gathered}
  m_{a, b}(H^k(A^n)) = m_{a, b}(H^k(A^{n-1})) + \langle W_{a, b}, W_{1, 0}\otimes H^{k-1}(A^{n-1}) \rangle + \langle W_{a, b}, (W_{0,0} \oplus W_{1, 1})\otimes H^{k-2}(A^{n-1}) \rangle \\
  + \langle W_{a, b}, W_{1, 0}\otimes H^{k-3}(A^{n-1}) \rangle + m_{a, b}(H^{k-4}(A^{n-1}))
\end{gathered}
\end{equation}
and
\begin{equation}\label{eqn:00-for-Hk}
\begin{gathered}
  m_{0,0}(H^k(A^n)) = m_{0,0}(H^k(A^{n-1})) + m_{1,0}(H^{k-1}(A^{n-1})) + m_{0,0}(H^{k-2}(A^{n-1}))\\ + m_{1,1}(H^{k-2}(A^{n-1})) + m_{1,0}(H^{k-3}(A^{n-1})) + m_{0,0}(H^{k-4}(A^{n-1})).
\end{gathered}
\end{equation}
\end{prop}
\begin{proof}
Apply $m_{a,b}$ to the isomorphism given by Lemma \ref{lem:X2-n-local-systems} to obtain Equation (\ref{eqn:ab-for-Hk}). 

For $j = 1, 3$, apply Lemma \ref{lem:rec-13} to compute
\[
  \langle W_{0,0}, W_{1,0} \otimes H^{k-j}(A^{n-1}) \rangle = m_{1, 0}(H^{k-j}(A^{n-1}))
\]
and apply Lemma \ref{lem:rec-2} to compute
\[
  \langle W_{0,0}, (W_{0,0} \oplus W_{1,1}) \otimes H^{k-2}(A^{n-1}) \rangle = m_{0,0}(H^{k-2}(A^{n-1})) + m_{1,1}(H^{k-2}(A^{n-1})).
\]
Plug in $a = b = 0$ to Equation (\ref{eqn:ab-for-Hk}) with these calculations to obtain Equation (\ref{eqn:00-for-Hk}).
\end{proof}

We can also more explicitly describe the representations $W_{a, b}$ that occur in $H^k(A^n)$. These descriptions are necessary to prove Theorem \ref{thm:low-degree-computation}. 
\begin{prop}\label{prop:mult-even-odd}
If $m_{a, b}(H^k(A^n)) \neq 0$, then $a + b \equiv k \pmod 2$ and $a + b \leq k$. For all $n\geq k$, all such $a \geq b \geq 0$ and $k \geq 0$ give $m_{a, b}(H^k(A^n)) \neq 0$. 
\end{prop}
\begin{proof}
To prove that if $m_{a, b}(H^k(A^n)) \neq 0$ then $a + b \equiv k \pmod 2$, we proceed by induction on $n$. For $n = 1$, the claim is true by Lemma \ref{lem:X2-local-systems}. Now assume the claim for $n-1$. Suppose $a + b \equiv k+1 \pmod 2$ or $a + b > k$ and let $V(k) = H^k(A^{n-1})$. Then using Lemma \ref{lem:rec-13} with $j = 1$ or $3$,
\begin{align*}
\langle W_{a, b}, W_{1, 0}\otimes V(k-j)) \rangle &= m_{a+1, b}(V(k-j)) + m_{a, b+1}(V(k-j)) \\
&\quad+ m_{a-1, b}(V(k-j)) + m_{a, b-1}(V(k-j)).
\end{align*}
Here, each summand is of the form $m_{\alpha, \beta}(V(k-j))$ for some $(\alpha, \beta) \in I(a,b)$ using the notation of the proof of Lemma \ref{lem:rec-13}. All such tuples satisfy $\alpha + \beta \equiv a+b+1 \pmod 2$, and so $\alpha + \beta \equiv k \pmod 2$. Since $k-j \equiv k+1 \pmod 2$, the inductive hypothesis shows that each $m_{\alpha, \beta}(V(k-j)) = 0$. Therefore, $\langle W_{a, b}, W_{1, 0}\otimes V(k-j)) \rangle = 0$.

By Lemma \ref{lem:rec-2} ,
\begin{align*}
\langle W_{a, b}, (W_{0,0} \oplus W_{1,1}) \otimes V(k-2)\rangle &= (2 - \ind_{\{a = b\}}(a, b)) m_{a,b}(V(k-2)) + m_{a-1, b-1}(V(k-2)) \\
&\quad + m_{a-1, b+1}(V(k-2)) + m_{a+1, b-1}(V(k-2)) + m_{a+1, b+1}(V(k-2)).
\end{align*}
Each summand is a multiple of $m_{\alpha, \beta}(V(k-2))$ for some $(\alpha, \beta) \in J(a,b)$ using the notation of the proof of Lemma \ref{lem:rec-2}. All such tuples satisfy $\alpha + \beta \equiv a+b \pmod 2$, and so $\alpha + \beta \equiv k+1 \pmod 2$. Since $\alpha + \beta \not\equiv k-2 \pmod 2$, the inductive hypothesis shows that each $m_{\alpha,\beta}(V(k-2)) = 0$. Therefore, $\langle W_{a, b}, (W_{0,0} \oplus W_{1,1}) \otimes V(k-2)\rangle = 0$.

By Proposition \ref{prop:recursive_mult}(\ref{eqn:ab-for-Hk}),
\begin{align*}
m_{a,b}(H^k(A^n)) &= m_{a, b}(V(k)) + \langle W_{a, b}, W_{1, 0}\otimes V(k-1)) \rangle + \langle W_{a, b}, (W_{0,0} \oplus W_{1, 1})\otimes V(k-2) \rangle \\
  &\quad+ \langle W_{a, b}, W_{1, 0}\otimes V(k-3) \rangle + m_{a, b}(V(k-4)) \\
  &= m_{a, b}(V(k)) + m_{a, b}(V(k-4)) = 0.
\end{align*}

For fixed $k \geq 0$ and $a \geq b \geq 0$ with $a+b \equiv k \pmod 2$ and $a+b \leq k$, we next show that $m_{a,b}(H^k(A^n)) \neq 0$ for $n \geq k$ by induction on $n$. For $n = 1$, the claim is again true by Lemma \ref{lem:X2-local-systems}. Assume the claim holds for $n-1$ and take any $k \leq n$. If $b \geq 1$, then by Proposition \ref{prop:recursive_mult}(\ref{eqn:ab-for-Hk}), Lemma \ref{lem:rec-13}, and the inductive hypothesis,
\[
  m_{a,b}(H^k(A^n)) \geq \langle W_{a,b}, W_{1,0} \otimes H^{k-1}(A^{n-1}) \rangle \geq m_{a,b-1}(H^{k-1}(A^{n-1})) > 0.
\]
If $a \neq 0$ and $b = 0$, then
\[
  m_{a,b}(H^k(A^n)) \geq \langle W_{a,b}, W_{1,0} \otimes H^{k-1}(A^{n-1}) \rangle \geq m_{a-1,b}(H^{k-1}(A^{n-1})) > 0.
\]
If $(a, b) = (0,0)$ and $k \geq 2$, then
\[
  m_{0,0}(H^k(A^n)) \geq \langle W_{0,0}, (W_{0,0} \oplus W_{1,1}) \otimes H^{k-2}(A^{n-1}) \rangle \geq (2 - \ind_{\{(1,1), (0,0)\}}(0,0)) m_{0,0}(H^{k-2}(A^{n-1})) > 0.
\]
If $k = 0$, then $H^0(A^n) = W_{0,0}$ for all $n$ and so $m_{0,0}(H^0(A^n)) = 1$.
\end{proof}

Throughout the rest of this section, let $\binom nm = 0$ if $n < m$. The following lemma contains the recursive calculations of $m_{a,b}(H^k(A^n))$ for select values of $a, b$, and $k$ which will be used in the proof of Theorem \ref{thm:low-degree-computation}, the main theorem of this section.
\begin{lem}\label{lem:small-cases-mult}
Let $a \geq b \geq 0$ and $M > 0$.
\begin{multicols}{2}
\begin{enumerate}[(a)]
  \item $m_{a,b}(H^0(A^M)) = \ind_{\{(0,0)\}}(a, b)$.\label{item:a,b,M,0}
  \item $m_{1,0}(H^1(A^M)) = M$.\label{item:1,0,M,1}
  \item $m_{0,0}(H^2(A^M)) = \binom{M+1}{2}$. \label{item:0,0,M,2}
  \item $m_{1,1}(H^2(A^M)) = \binom{M+1}{2}$. \label{item:1,1,M,2}
  \item $m_{2,0}(H^2(A^M)) = \binom M 2$. \label{item:2,0,M,2}
  \item $m_{1,0}(H^3(A^M)) = \binom M3 + 2 \binom{M+1}{3} + M^2$. \label{item:1,0,M,3}
  \item $m_{0,0}(H^4(A^M)) = \frac{M(M+1)(M^2 + M + 2)}{8}$.\label{item:0,0,M,4}
\end{enumerate}
\end{multicols}
\end{lem}
\begin{proof}
All proofs are by induction on $M$. We can check manually that all claims hold for $M = 1$. Assume that they hold for $M -1$.
\begin{enumerate}
  \item For all $M$, $H^0(A^M) = W_{0,0}$, the trivial $\Sp(4, \Q_\ell)$-representation.

  \item By Proposition \ref{prop:recursive_mult}(\ref{eqn:ab-for-Hk}), Lemma \ref{lem:rec-13} and (\ref{item:a,b,M,0}),
  \begin{align*}
    m_{1,0}&(H^1(A^M)) = m_{1,0}(H^1(A^{M-1})) + \langle W_{1,0}, W_{1,0} \otimes H^0(A^{M-1}) \rangle \\
    &= (M-1) + \left(m_{2,0}(H^0(A^{M-1})) + m_{1,1}(H^0(A^{M-1})) + m_{0,0}(H^0(A^{M-1})) \right) = M.
  \end{align*}

  \item By Proposition \ref{prop:recursive_mult}(\ref{eqn:00-for-Hk}), (\ref{item:a,b,M,0}), and (\ref{item:1,0,M,1}),
  \begin{align*}
  m_{0,0}(H^2(A^{M})) &= m_{0,0}(H^2(A^{M-1})) + m_{1,0}(H^1(A^{M-1})) + m_{0,0}(H^0(A^{M-1}))+ m_{1,1}(H^0(A^{M-1})) \\
  &= \binom{M}{2} + (M-1) + 1 = \binom{M+1}{2}.
  \end{align*}

  \item By Proposition \ref{prop:recursive_mult}(\ref{eqn:ab-for-Hk}), Lemma \ref{lem:rec-13}, Lemma \ref{lem:rec-2}, (\ref{item:a,b,M,0}), and (\ref{item:1,0,M,1}),
  \begin{align*}
  m_{1,1}(H^2(A^{M})) &= m_{1,1}(H^2(A^{M-1})) + \langle W_{1,1}, W_{1,0} \otimes (H^1(A^{M-1})) \rangle + \langle W_{1,1}, (\Q_\ell \oplus W_{1,1}) \otimes (H^0(A^{M-1})) \rangle \\
  &= \binom{M}{2} + \left(m_{2, 1}(H^1(A^{M-1})) + m_{1, 0}(H^1(A^{M-1}))\right) \\
  &\quad+ \left(m_{1,1}(H^0(A^{M-1})) + m_{0,0}(H^0(A^{M-1})) + m_{2, 0}(H^0(A^{M-1})) + m_{2,2}(H^0(A^{M-1}))\right) \\
  &= \binom{M}{2} + \left(m_{2, 1}(H^1(A^{M-1})) + (M-1)\right) + \left(1\right) = \binom{M+1}{2}
  \end{align*}
  where the last equality follows by Proposition \ref{prop:mult-even-odd}, which gives that $m_{2,1}(H^1(A^{M-1})) = 0$.

  \item By Proposition \ref{prop:recursive_mult}(\ref{eqn:ab-for-Hk}), Lemma \ref{lem:rec-13}, Lemma \ref{lem:rec-2}, (\ref{item:a,b,M,0}), and (\ref{item:1,0,M,1}),
  \begin{align*}
  m_{2,0}(H^2(A^{M})) &= m_{2,0}(H^2(A^{M-1})) + \langle W_{2,0}, W_{1,0} \otimes (H^1(A^{M-1})) \rangle + \langle W_{2,0},(\Q_\ell \oplus W_{1,1}) \otimes (H^0(A^{M-1})) \rangle \\
  &= \binom{M-1}{2} + \left(m_{3, 0}(H^1(A^{M-1})) + m_{2, 1}(H^1(A^{M-1})) + m_{1, 0}(H^1(A^{M-1}))\right) \\
  &\quad+ \left(2m_{2, 0}(H^0(A^{M-1})) + m_{1,1}(H^0(A^{M-1})) + m_{3, 1}(H^0(A^{M-1}))\right) \\
  &=\binom{M-1}{2} + (M-1) = \binom M 2.
  \end{align*}
  where the last equality follows by Proposition \ref{prop:mult-even-odd}, which gives that $m_{3,0}(H^1(A^{M-1})) = m_{2,1}(H^1(A^{M-1})) = 0$.

  \item By Proposition \ref{prop:recursive_mult}(\ref{eqn:ab-for-Hk}), Lemma \ref{lem:rec-13}, Lemma \ref{lem:rec-2}, and (\ref{item:a,b,M,0}) - (\ref{item:2,0,M,2}),
  \begin{align*}
  m_{1,0}(H^3(A^{M})) &= m_{1,0}(H^3(A^{M-1})) + \langle W_{1,0}, W_{1,0} \otimes H^2(A^{M-1}) \rangle \\
  &\quad+ \langle W_{1,0}, (\Q_\ell \oplus W_{1,1})\otimes H^1(A^{M-1}) \rangle + \langle W_{1,0}, W_{1,0} \otimes H^0(A^{M-1}) \rangle \\
  &= m_{1,0}(H^3(A^{M-1})) + \left(m_{2, 0}(H^2(A^{M-1})) + m_{1,1}(H^2(A^{M-1})) + m_{0,0}(H^2(A^{M-1}))\right) \\
  &\quad+ \left(2m_{1,0}(H^1(A^{M-1})) + m_{2,1}(H^1(A^{M-1}))\right) \\
  &\quad+ \left(m_{2, 0}(H^0(A^{M-1})) + m_{1,1}(H^0(A^{M-1})) + m_{0,0}(H^0(A^{M-1}))\right) \\
  &= \left(\binom{M-1}{3} + 2\binom M 3 + (M-1)^2 \right) + \left(\binom{M-1}{2} +\binom{M}{2} + \binom{M}{2}\right) \\
  &\quad + \left(2(M-1) + m_{2,1}(H^1(A^{M-1}))\right) + 1\\
  &= \binom M3 + 2 \binom{M+1}{3} + M^2
  \end{align*}
  where again, the last equality uses that $m_{2,1}(H^1(A^{M-1})) = 0$ by Proposition \ref{prop:mult-even-odd}.

  \item By Proposition \ref{prop:recursive_mult}(\ref{eqn:00-for-Hk}), (\ref{item:a,b,M,0}) - (\ref{item:1,1,M,2}), and (\ref{item:1,0,M,3}), 
  \begin{align*}
    m_{0,0}(H^4(A^{M})) &= m_{0,0}(H^4(A^{M-1})) + m_{1,0}(H^3(A^{M-1})) + m_{0,0}(H^2(A^{M-1})) + m_{1,1}(H^2(A^{M-1})) \\
    &\quad+ m_{1,0}(H^1(A^{M-1})) + m_{0,0}(H^0(A^{M-1})) \\
    &= \frac{(M-1)(M)(M^2 - M + 2)}{8} + \left(\binom{M-1}{3} + 2\binom M3 + (M-1)^2\right) \\
    &\quad + \binom{M}{2} + \binom{M}{2} + (M-1) + 1 \\
    &= \frac{M(M+1)(M^2 + M + 2)}{8}. \qedhere
  \end{align*}
\end{enumerate}
\end{proof}

We are now ready to compute $H^k(\mathcal X_2^n; \Q_\ell)$ for all $n \geq 1$ and $0 \leq k \leq 5$.
\begin{thm:low-degree-computation}
For all $n \geq 1$, 
\begin{align*}
H^k(\mathcal X_2^n; \Q_\ell) &= \begin{cases}
  \Q_\ell & k = 0 \\
  0 & k = 1, 3 \\
  \left(\binom{n+1}{2} + 1\right) \Q_\ell(-1) & k = 2 \\
  \left(\frac{n(n+1)(n^2+n+2)}{8} + \binom{n+1}{2}\right)\Q_\ell(-2) & k = 4\\
  \binom{n+1}{2} \Q_\ell(-5) \oplus \binom n2 \Q_\ell(-4) & k = 5
\end{cases}
\end{align*}
up to semi-simplification.
\end{thm:low-degree-computation}
\begin{proof}
Recall that the $(p, q)$-entry of the $E_2$-sheet of the Leray spectral sequence of $\pi^n: \mathcal X_2^n \to \mathcal A_2$ is $H^p(\mathcal A_2; H^q(A^n))$ by Proposition \ref{prop:leray}; denote this entry by $E_2^{p,q}(n)$. We compute many entries on the $E_2$-sheet and list the nonzero results in Figure \ref{fig:X2-n-spectral-sequence}, from which the theorem follows directly. The special case of $n = 1$ is given in Figure \ref{fig:X2-spectral-sequence}. All computations here are up to semi-simplification. 
\begin{enumerate}
  \item $E_2^{p,0}(n) = H^p(\mathcal A_2; \Q_\ell)$ for all $p$.
  \begin{proof}
  By definition and Lemmas \ref{lem:X2-n-local-systems} and \ref{lem:small-cases-mult},
  \[
    E_2^{p,0}(n) = H^p(\mathcal A_2; H^0(A^{n})) = \bigoplus_{a \geq b \geq 0} m_{a, b}(H^0(A^{n}))H^p\left(\mathcal A_2; \mathbf V_{a, b} \left(\frac{a+b}{2}\right)\right) = H^p(\mathcal A_2; \Q_\ell).\qedhere
  \]
  \end{proof}

  \item $E_2^{p, q}(n) = 0$ for all $p \geq 0$, $q \equiv 1 \pmod 2$.
  \begin{proof}
  Suppose $m_{a, b}(H^q(A^n)) \neq 0$. By Proposition \ref{prop:mult-even-odd}, $a+b \equiv 1 \pmod 2$. Therefore
  \[
    E_2^{p,q}(n) = H^p(\mathcal A_2; H^q(A^{n})) = \bigoplus_{\substack{a\geq b \geq 0 \\ a+b \equiv 1 \pmod 2}} m_{a, b}(H^q(A^{n})) H^p(\mathcal A_2; \mathbf V_{a, b})\left(\frac{a+b-q}{2}\right) = 0
  \]
  where the last equality follows since $H^p(\mathcal A_2; \mathbf V_{a, b}) = 0$. (See the remark before Theorem \ref{thm:petersen}.)
  \end{proof}

  \item $E_2^{1, q}(n) = 0$ for all $q \geq 0$.
  \begin{proof}
  By definition and Corollary \ref{cor:petersen}(\ref{eqn:cor-petersen-deg1}),
  \[
    E_2^{1,q}(n) = H^1(\mathcal A_2; H^q(A^n)) = \bigoplus_{a \geq b \geq 0} m_{a,b}(H^q(A^n)) H^1(\mathcal A_2; \mathbf V_{a,b}) = 0. \qedhere
  \]
  \end{proof}

  \item $E_2^{p, 2}(n) = \binom{n+1}{2} H^p(\mathcal A_2; \Q_\ell)(-1)$ and $E_2^{p, 4}(n) = \frac{n(n+1)(n^2+n+2)}{8} H^p(\mathcal A_2; \Q_\ell)(-2)$ for $p = 0, 1, 2$.
  \begin{proof}
  Let $q = 2$, $4$. By definition and Proposition \ref{prop:mult-even-odd},
  \[
    E_2^{p, q}(n) = H^p(\mathcal A_2; H^q(A^{n})) = \bigoplus_{\substack{a + b \leq q \\ a + b \equiv 0 \pmod 2}} m_{a,b}(H^q(A^{n}))H^p(\mathcal A_2; \mathbf V_{a, b})\left(\frac{a+b-q}{2}\right).
  \]
  By Theorem \ref{thm:petersen}, $H^p(\mathcal A_2; \mathbf V_{a, b}) = 0$ for $(a, b) \neq 0$ and $p = 0$, $1$, so
  \[
    E_{2}^{p, q}(n) = m_{0,0}(H^q(A^{n}))H^p(\mathcal A_2;\Q_\ell) \left(-\frac{q}{2}\right)
  \]
  for $p = 0$, $1$. The claim then follows by Lemma \ref{lem:small-cases-mult}. For $p = 2$, Theorem \ref{thm:petersen} gives that $H^2(\mathcal A_2; \mathbf V_{a, b}) \neq 0$ only if $a = b$ even, and so
  \[
    E_2^{2, q}(n) = m_{0,0}(H^q(A^{n}))H^2(\mathcal A_2; \Q_\ell)\left(-\frac q2\right) \oplus \begin{cases} 
    0 & q = 2\\
    m_{2,2}(H^4(A^{n}))H^2(\mathcal A_2; \mathbf V_{2,2}) & q = 4.
    \end{cases}
  \]
  In the case $q = 2$, the claim then follows from Lemma \ref{lem:small-cases-mult}. In the case $q = 4$, the claim follows from Corollary \ref{cor:petersen}(\ref{eqn:cor-petersen-2,2}) since $H^2(\mathcal A_2; \mathbf V_{2,2}) = 0$. 
  \end{proof}

  \item $E_2^{3, 2}(n) = \binom{n+1}{2} \Q_\ell(-5) \oplus \binom n 2 \Q_\ell(-4)$ and $E_2^{4, 2}(n) = 0$.
  \begin{proof}
  Let $p = 3$, $4$. By definition, Proposition \ref{prop:mult-even-odd}, Theorem \ref{thm:a2}, and Lemma \ref{lem:small-cases-mult},
  \begin{align*}
    E_2^{p, 2}(n) &= H^p(\mathcal A_2; H^2(A^{n})) = \bigoplus_{a \geq b \geq 0} m_{a, b}(H^2(A^{n})) H^p(\mathcal A_2; \mathbf V_{a, b})\left(\frac{a + b - 2}{2}\right)\\
    &= m_{0,0}(H^2(A^{n})) H^p(\mathcal A_2; \mathbf V_{0,0})(-1) \oplus m_{1,1}(H^2(A^{n})) H^p(\mathcal A_2; \mathbf V_{1,1}) \oplus m_{2,0}(H^2(A^{n})) H^p(\mathcal A_2; \mathbf V_{2,0})\\
    &= \binom{n+1}{2} H^p(\mathcal A_2; \mathbf V_{1,1}) \oplus \binom{n}{2} H^p(\mathcal A_2; \mathbf V_{2,0}).
  \end{align*}
  Then the claim follows from Corollary \ref{cor:petersen}(\ref{eqn:cor-petersen-1,1}) and (\ref{eqn:cor-petersen-2,0}).
  \end{proof}

  \item $E_2^{p, q}(n) = 0$ for all $p \geq 4$ and $q = 0, 1, 2, 3$.
  \begin{proof}
  By definition and Proposition \ref{prop:mult-even-odd},
  \[
    E_2^{p,q}(n) = H^p(\mathcal A_2; H^q(A^n)) = \bigoplus_{\substack{a+b \leq q \\ a + b \equiv q \pmod 2}} m_{a,b}(H^q(A^n)) H^p(\mathcal A_2; \mathbf V_{a,b}) \left(\frac{a+b-q}{2}\right).
  \]
  For $q = 0, 1, 2, 3$, the only tuples $(a, b)$ with $a \geq b \geq 0$ and $a+b \leq q$ satisfy $(a,b) \in \{(0,0), (1,0), (1,1), (2,0), (2,1)\}$. By the remark before Theorem \ref{thm:petersen}, $H^p(\mathcal A_2; \mathbf V_{a,b}) = 0$ for all $p \geq 0$ if $a+b \equiv 1 \pmod 2$. By Corollary \ref{cor:petersen}(\ref{eqn:cor-petersen-1,1}) and (\ref{eqn:cor-petersen-2,0}), $H^p(\mathcal A_2; \mathbf V_{a,b}) = 0$ for all $p \geq 4$ and $(a,b) = (1,1), (2,0)$. By Theorem \ref{thm:a2}, $H^p(\mathcal A_2; \mathbf V_{0,0}) = 0$ for all $p \geq 3$. Therefore, all summands of the direct sum above are zero.
  \end{proof}
\end{enumerate}
\end{proof}
\begin{rmk}
Theorem \ref{thm:low-degree-computation} is consistent with the stabilization result \cite[Theorem 6.1]{grushevsky--hulek--tommasi}, which says that the rational cohomology of $\mathcal X_g^n$ stabilizes in degrees $k < g = 2$. 
\end{rmk}
\begin{figure}
\centering
\includegraphics[width=\textwidth]{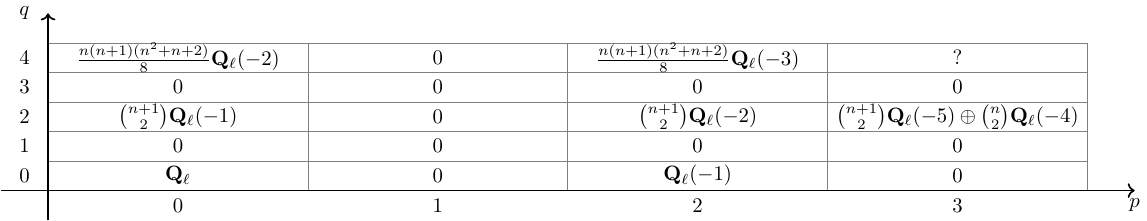}
\caption{Some low degree terms of the $E_2$-page of the Leray spectral sequence for $\pi^n: \mathcal X_2^n\to \mathcal A_2$ determined in Theorem \ref{thm:low-degree-computation}.}
\label{fig:X2-n-spectral-sequence}
\end{figure}

\subsection{Explicit computations for $n = 2$.}
Once one has computed $m_{a, b}(H^k(A^{n}))$ for fixed $n$ and for all $k \geq 0$, $a \geq b \geq 0$, one can in theory apply Proposition \ref{prop:leray} and Theorem \ref{thm:petersen} to determine $H^*(\mathcal X_2^n; \Q_\ell)$. In this subsection, we detail the results of this process for $n = 2$. 

\begin{lem}\label{lem:X2-2-local-systems}
The local systems $H^k(A^2; \Q_\ell)$ on $\mathcal A_2$ are
\[
  H^k(A^2; \Q_\ell) \cong \begin{cases}
  \mathbf V_{0,0} & k = 0 \\
  2\mathbf V_{1, 0} & k = 1 \\
  3 \mathbf V_{0,0}(-1) \oplus 3 \mathbf V_{1,1} \oplus \mathbf V_{2,0} & k = 2 \\
  6 \mathbf V_{1, 0}(-1) \oplus 2\mathbf V_{2, 1} & k = 3\\
  6 \mathbf V_{0,0}(-2) \oplus 4\mathbf V_{1,1}(-1) \oplus 3\mathbf V_{2, 0}(-1) \oplus \mathbf V_{2,2} & k = 4 \\
  6 \mathbf V_{1, 0}(-2) \oplus 2\mathbf V_{2, 1}(-1) & k = 5 \\
  3 \mathbf V_{0,0}(-3) \oplus 3 \mathbf V_{1,1}(-2) \oplus \mathbf V_{2,0}(-2) & k = 6 \\
  2\mathbf V_{1,0}(-3) & k = 7 \\
  \mathbf V_{0,0}(-4) & k = 8 \\
  0 & k > 8.
  \end{cases}
\]
\end{lem}
\begin{proof}
This is a direct computation using Lemma \ref{lem:X2-n-local-systems}, Lemma \ref{lem:rec-13}, Lemma \ref{lem:rec-2}, and Lemma \ref{lem:X2-local-systems}.
\end{proof}

Using Lemma \ref{lem:X2-2-local-systems}, we can compute all entries of the $E_2$-page of the Leray spectral sequence for $\pi^2: \mathcal X_2^2 \to \mathcal A_2$. As always, the following results are up to semi-simplification.

\begin{lem}\label{lem:e2-terms-X2-2}
For $q = 1$, $3$, $5$, $7$, and all $p$,
\[
  H^p(\mathcal A_2; H^q(A^{2})) = 0.
\]
For $q = 0$, $8$,
\[
  H^p(\mathcal A_2; H^q(A^{2})) \cong \begin{cases}
  \Q_\ell\left(-\frac q2\right) & p = 0 \\
  \Q_\ell\left(-\frac q2-1\right) & p = 2 \\
  0 & \text{otherwise}.
  \end{cases}
\]
For $q = 2$, $6$,
\[
  H^p(\mathcal A_2; H^q(A^{2})) \cong \begin{cases}
  3\Q_\ell\left(-\frac q2\right) & p = 0\\
  3\Q_\ell\left(-\frac q2-1\right) & p = 2 \\
  3\Q_\ell\left(-\frac{q-2}{2}-5\right) \oplus \Q_\ell\left(-\frac{q-2}{2}-4\right) & p = 3 \\
  0 & \text{otherwise}.
  \end{cases} 
\]
For $q = 4$,
\[
  H^p(\mathcal A_2; H^4(A^{2})) \cong \begin{cases}
  6\Q_\ell(-2) & p = 0 \\
  6\Q_\ell(-3) & p = 2 \\
  3\Q_\ell(-5) \oplus 4\Q_\ell(-6) & p = 3 \\
  0 & \text{otherwise}.
  \end{cases}
\]
\end{lem}
\begin{proof}
By definition, $E^{p,q}_2 = H^p(\mathcal A_2; H^q(A^2)) = \bigoplus_{a \geq b \geq 0} m_{a, b}(H^q(A^{2}))H^p(\mathcal A_2; \mathbf V_{a, b})$. Suppose $q$ is odd. By Proposition \ref{prop:mult-even-odd}, if $m_{a, b}(H^q(A^{2})) \neq 0$, then $a+b \equiv 1 \pmod 2$. For such $(a, b)$, the remarks before Theorem \ref{thm:petersen} imply that $H^p(\mathcal A_2; \mathbf V_{a, b}) = 0$ for all $p$.

For $0 \leq q\leq 8$ even, combine Lemma \ref{lem:X2-n-local-systems}, Lemma \ref{lem:X2-2-local-systems}, Theorem \ref{thm:a2}, and Corollary \ref{cor:petersen}.
\end{proof}
\begin{rmk}
The computations in Lemma \ref{lem:e2-terms-X2-2} are consistent with Theorem \ref{thm:low-degree-computation}.
\end{rmk}

By the usual argument, we obtain the following theorem using these preliminaries.
\begin{thm}\label{thm:2-power-universal-surface}
The cohomology of the second fiber power $\mathcal X_2^2$ of the universal abelian surface is given by

\[
  H^k(\mathcal X_2^2;\Q_\ell) = \begin{cases}
  \Q_\ell & k = 0 \\
  0 & k = 1, 3, k > 10 \\
  4\Q_\ell(-1) & k = 2 \\
  9\Q_\ell(-2) & k = 4 \\
  3\Q_\ell(-5) \oplus \Q_\ell(-4) & k = 5 \\
  9\Q_\ell(-3) & k = 6 \\
  3\Q_\ell(-5) \oplus 4\Q_\ell(-6) & k = 7 \\
  4\Q_\ell(-4) & k = 8 \\
  3\Q_\ell(-7) \oplus \Q_\ell(-6) & k = 9 \\
  \Q_\ell(-5) & k = 10\\
  \end{cases}
\]
up to semi-simplification.
\end{thm}
\begin{proof}
By Proposition \ref{prop:leray}, there is a spectral sequence
\[
  E_2^{p,q} = H^p(\mathcal A_2; H^q(A^{2})) \implies H^{p+q}(\mathcal X_2^2; \Q_\ell)
\]
which degenerates on the $E_2$-page. Combining the lemmas in this section gives the entries of the $E_2$-page as recorded in Figure \ref{fig:X2-2-spectral-sequence}.
\end{proof}

\begin{figure}
\centering
\includegraphics{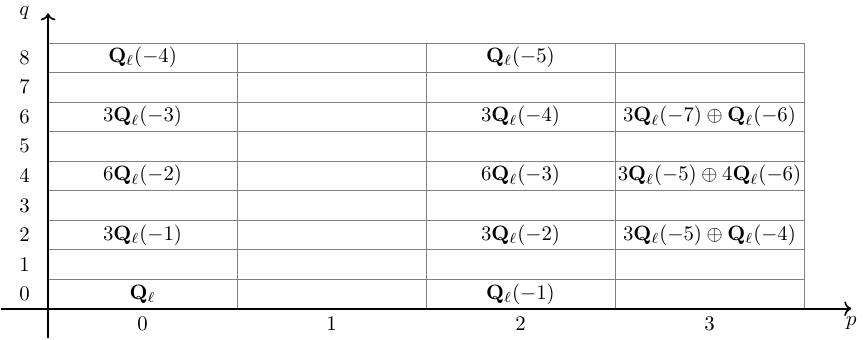}
\caption{The nonzero terms on the $E_2$-page of the Leray spectral sequence for $\pi^2: \mathcal X_2^2 \to \mathcal A_2$.}
\label{fig:X2-2-spectral-sequence}
\end{figure}

\section{Cohomology of $\mathcal X_2^{\Sym(n)}$}\label{sec:symmetric-powers}
In this section, we compute $H^*(\mathcal X_2^{\Sym(n)}; \Q_\ell)$. As opposed to Section \ref{sec:fiber-powers}, we show that the cohomology in fixed degree stabilizes as $n$ increases and explicitly give the computations for small degree. Afterwards, we give a complete description of the cohomology for the case $n = 2$. For brevity, we will often drop the constant coefficients if the context is clear, writing $H^k(A^n)$ instead of $H^k(A^n;\Q_\ell)$.

\subsection{Computations for general $n$.}\label{sec:sym-n}
Let $\pi^n: \mathcal X_2^{\Sym(n)} \to \mathcal A_2$ and $\pi: \mathcal X_2 \to \mathcal A_2$. Let $A$ be an abelian surface. For all $n \geq 1$, the symmetric group $S_n$ acts on $A^n$ by permuting the coordinates, which induces an action of $S_n$ on $H^*(A^n)$. The $S_n$-action on $A^n$ is not free but this is not a problem in the context of stacks. On the other hand, 
\[
  H^k(A^n; \Q_\ell) \cong \bigoplus_{\substack{(k_1, \dots, k_n)\\ \sum k_i = k}} \bigotimes_{i=1}^n H^{k_i}(A; \Q_\ell)
\]
by the K\"unneth formula. 
\begin{lem}\label{lem:sn-fixed}
There are isomorphisms of local systems 
\[
  H^k(\Sym^n A; \Q_\ell) \cong H^k(A^n; \Q_\ell)^{S_n} \cong \bigoplus_{\substack{(k_0, \dots, k_4) \\ \sum_{i=0}^4 ik_i = k \\ \sum_{i=0}^4 k_i = n}} \bigwedge^{k_1} H^1(A) \otimes \Sym^{k_2}H^2(A) \otimes \bigwedge^{k_3} H^3(A)\otimes \Q_\ell(-2k_4).
\]
\end{lem}
\begin{proof}
For any abelian surface $A$, there is an isomorphism $H^k(\Sym^n A; \Q_\ell) \to H^k(A^n; \Q_\ell)^{S_n}$ by the Hochschild--Serre spectral sequence for $A^n \to A^n/S_n \cong \Sym^n A$. For any $\sigma \in S_n$, let $Q_\sigma(x_1, \dots, x_n)$ be the sum of all products $x_i x_j$ with $i < j$ which occur in reversed order in the sequence $\sigma^{-1}(1), \dots, \sigma^{-1}(n)$. The induced action of $S_n$ on $H^k(A^n; \Q_\ell)$ is then given by 
\[
  \sigma\cdot (c_1 \otimes \dots \otimes c_n) = (-1)^{\varepsilon}\left(c_{\sigma^{-1}(1)} \otimes \dots \otimes c_{\sigma^{-1}(n)}\right) \in  \bigotimes_{i=1}^n H^{m_{\sigma^{-1}(i)}}(A; \Q_\ell) \subseteq H^k(A^n; \Q_\ell)
\]
for all $c_1 \otimes \dots \otimes c_n \in \bigotimes_{i=1}^n H^{m_i}(A; \Q_\ell)$ with $\varepsilon = Q_\sigma(m_1, \dots, m_n)$ as described in \cite{macdonald}. For example if $\sigma= (\ell, \ell+1)$ is a transposition, then $Q_\sigma(x_1, \dots, x_n) = x_{\ell}x_{\ell+1}$ which encodes the fact that $H^*(A^n; \Q_\ell)$ is a graded commutative ring with respect to the cup product and that for any $c_1, \dots, c_n \in H^*(A; \Q_\ell)$ with $c_i \in H^{k_i}(A; \Q_\ell)$ for some $k_i \geq 0$ for all $1 \leq i \leq n$, the cup product $c_1 \smile \dots \smile c_n$ corresponds to the simple tensor $c_1 \otimes \dots \otimes c_n$ under the K\"unneth isomorphism. There is a relationship between $Q_{\sigma_1}$, $Q_{\sigma_2}$, and $Q_{\sigma_1\sigma_2}$ which encodes the fact that $S_n$ acts on $H^k(A^n; \Q_\ell)$; we refer the reader to \cite[(1.1)-(1.3)]{macdonald} for more properties of the polynomials $Q_\sigma$ since we do not use any of them explicitly in this proof.

There is a projection $p: H^k(A^n; \Q_\ell) \to H^k(A^n; \Q_\ell)^{S_n}$ given by averaging. For each fixed $(m_1, \dots, m_n)$ with $\sum_{i=1}^n m_i = k$ and $m_1 \geq \dots \geq m_n$, consider the $S_n$-subrepresentation 
\[
  W_{m_1, \dots, m_n} := \bigoplus_{\sigma \in S_n} \bigotimes_{i=1}^n H^{m_{\sigma(i)}}(A; \Q_\ell) \subseteq H^k(A^n; \Q_\ell).
\]
The summand corresponding to $1 \in S_n$ above can be written as $\bigotimes_{i=0}^4 H^i(A; \Q_\ell)^{\otimes k_i}$ with $k_i = \#\{m_j: m_j = i\}$. For any simple tensor $c \in \bigotimes_{i=1}^n H^{m_{\sigma(i)}}(A; \Q_\ell)$ in $W_{m_1, \dots, m_n}$, it is straightforward to check that $\sigma \cdot c \in \bigotimes_{i=0}^4 H^i(A; \Q_\ell)^{\otimes k_i}$. Because $W_{m_1, \dots, m_n}$ is spanned over $\Q_\ell$ by such simple tensors $c$ and $p(c) = p(\sigma\cdot c)$ for all $\sigma \in S_n$, the image $p(W_{m_1, \dots, m_n})$ is spanned by the images $p(c)$ of simple tensors $c \in \bigotimes_{i=0}^4 H^i(A; \Q_\ell)^{\otimes k_i}$. Therefore, $p$ restricted to $\bigotimes_{i=0}^4 H^i(A; \Q_\ell)^{\otimes k_i}$ is surjective onto $p(W_{m_1, \dots, m_n})$ with kernel $\langle c - \sigma \cdot c: \sigma \in \prod_{i=0}^4 S_{k_i} \leq S_n\rangle$.

For a transposition $\sigma = (\ell_1 \ell_2) \in S_{k_j} \leq \prod_{i=0}^4 S_{k_i}$ with $\ell_1 > \ell_2$, observe that
\[
  Q_\sigma(m_1, \dots, m_n) = m_{\ell_2}m_{\ell_1}+\sum_{\ell = \ell_2+1}^{\ell_1 - 1} \left(m_{\ell}m_{\ell_1}+m_{\ell_2}m_{\ell}\right)  = j^2 + \sum_{\ell = \ell_2+1}^{\ell_1 - 1} 2j^2 \equiv j \pmod 2.
\]
This implies that for $i \equiv 0 \pmod 2$, 
\[
  H^i(A; \Q_\ell)^{\otimes k_i}/\langle c - \sigma \cdot c: \sigma \in S_{k_i} \rangle \cong \Sym^{k_i} H^i(A; \Q_\ell)
\]
and for $i \equiv 1 \pmod 2$,
\[
  H^i(A; \Q_\ell)^{\otimes k_i}/\langle c - \sigma \cdot c: \sigma \in S_{k_i} \rangle \cong \bigwedge^{k_i} H^i(A; \Q_\ell).
\]
Combining all of the above,
\begin{align*}
  p(W_{m_1, \dots, m_n}) &\cong \bigotimes_{i=0}^4 \left(H^i(A; \Q_\ell)^{\otimes k_i}/\langle c - \sigma \cdot c: \sigma \in S_{k_i} \rangle\right) \\
  &\cong \bigwedge^{k_1} H^1(A; \Q_\ell) \otimes \Sym^{k_2}H^2(A; \Q_\ell) \otimes \bigwedge^{k_3} H^3(A; \Q_\ell) \otimes \Q_\ell.
\end{align*}
Therefore, we have proven the desired isomorphisms on the level of $\Sp(4, \Q_\ell)$-representations. To determine the structure as $\GSp(4, \Q_\ell)$-representations and as local systems, we add in appropriate Tate twists as in the proofs of Lemma \ref{lem:X2-local-systems} and \ref{lem:X2-n-local-systems}. 
\end{proof}
\begin{lem}
For fixed $k$, $H^k(\mathcal X_2^{\Sym(n)}; \Q_\ell)$ stabilizes for $n \geq k$ up to semi-simplification.
\end{lem}
\begin{proof}
Fix $k\in \N$. For each $n \in \N$, consider the set 
\[
  S(n) := \left\{(k_0, \dots, k_4) \in \N^5 : \sum_{i=0}^4 ik_i = k, \sum_{i=0}^4 k_i = n\right\}.
\] 
If $n \geq k$, then there is a bijection $S(k) \to S(n)$ given by sending each $(k_0, \dots, k_4) \mapsto (k_0 + (n-k), k_1, \dots, k_4)$. Using this bijection and the fact that $\Sym^m H^0(A; \Q_\ell) \cong \Q_\ell$ for any $m \geq 0$, compute for all $n \geq k$ that
\begin{align*}
H^k(A^k; \Q_\ell)^{S_k} &\cong \bigoplus_{(k_0, \dots, k_4) \in S(k)} \bigwedge^{k_1}H^1(A; \Q_\ell) \otimes \Sym^{k_2} H^2(A; \Q_\ell) \otimes \bigwedge^{k_3}H^3(A; \Q_\ell)\otimes \Q_\ell(-2k_4) \\
&\cong \bigoplus_{(k_0, \dots, k_4) \in S(n)} \bigwedge^{k_1}H^1(A; \Q_\ell) \otimes \Sym^{k_2} H^2(A; \Q_\ell) \otimes \bigwedge^{k_3}H^3(A; \Q_\ell)\otimes \Q_\ell(-2k_4) \cong H^k(A^n; \Q_\ell)^{S_n}.
\end{align*}
By Lemma \ref{lem:sn-fixed} and the above computation, there is an isomorphism of local systems
\[
  H^k(\Sym^k A; \Q_\ell) \cong H^k(\Sym^n A; \Q_\ell)
\]
for all $n \geq k$. 

Up to semi-simplification,
\[
  H^k(\mathcal X_2^{\Sym(n)}; \Q_\ell) = \bigoplus_{p+q=k} H^p(\mathcal A_2; H^q(\Sym^nA; \Q_\ell)) = \bigoplus_{p+q=k} H^p(\mathcal A_2; H^q(\Sym^k A; \Q_\ell))
\]
where the first equality follows by Proposition \ref{prop:leray} and the second equality follows by the first part of the proof which shows that as local systems, 
\[
  H^q(\Sym^q A; \Q_\ell) \cong H^q(\Sym^k A; \Q_\ell) \cong H^q( \Sym^n A; \Q_\ell) 
\]
for all $k, n$ such that $q \leq k \leq n$.
\end{proof}

For small $k$, this simplifies the computations for the local systems $H^k(\Sym^n A; \Q_\ell)$. 
\begin{prop}\label{prop:example-sym-n}
For all $n \geq 1$,
\begin{align*}
H^0(\Sym^nA; \Q_\ell) &\cong \Q_\ell,\\
H^1(\Sym^nA; \Q_\ell) &\cong \mathbf V_{1,0}.
\end{align*}
For all $n \geq 2$, 
\[
  H^2(\Sym^nA; \Q_\ell) \cong 2\Q_\ell(-1) \oplus 2\mathbf V_{1,1}.
\]
For all $n \geq 3$,
\[
  H^3(\Sym^nA; \Q_\ell) \cong 4\mathbf V_{1,0}(-1) \oplus \mathbf V_{2,1}.
\]
For all $n \geq 4$, 
\[
  H^4(\Sym^nA; \Q_\ell) \cong 7\Q_\ell(-2) \oplus 4\mathbf V_{1,1}(-1) \oplus 2 \mathbf V_{2,0}(-1) \oplus 2\mathbf V_{2,2}.
\]
For all $n \geq 5$, 
\[
  H^5(\Sym^nA; \Q_\ell) \cong 10 \mathbf V_{1,0}(-2) \oplus 5 \mathbf V_{2,1}(-1) \oplus \mathbf V_{3,2}.
\]
\end{prop}
\begin{proof}
The necessary facts from the representation theory of $\Sp(4, \Q_\ell)$ are Lemma \ref{lem:rep-thy-1,0}, Lemma \ref{lem:rep-thy-1,1}, and \cite[Exercise 16.11]{fulton--harris} which says that $\Sym^a W_{1,1} \cong \bigoplus_{k=0}^{\lfloor \frac a2 \rfloor} W_{a - 2k, a - 2k}$ for any $a \geq 0$. Applying these facts to the direct sum given by Lemma \ref{lem:sn-fixed} gives the decomposition into irreducible $\Sp(4, \Q_\ell)$-representations as claimed. Finally, add appropriate Tate twists as in the proof of Lemma \ref{lem:X2-local-systems}. 

As an example, we work out the computation for $H^4(\Sym^n A; \Q_\ell)$ for $n \geq 4$ explicitly. For $k = 4$, the tuples $(k_0, \dots, k_4) \in \N^5$ satisfying $\sum_{i=0}^4 ik_i = k = 4$ and $\sum_{i=0}^4 k_i = n$ are
\[
  (n-1, 0, 0, 0, 1), \, (n-2, 1, 0, 1, 0), \, (n-2, 0, 2, 0, 0), \, (n-3, 2, 1, 0, 0), \, (n-4, 4, 0, 0, 0).
\]
Lemma \ref{lem:sn-fixed} gives 
\[
  H^4(\Sym^n A; \Q_\ell) \cong H^4(A) \oplus \left(H^1(A) \otimes H^3(A)\right) \oplus \Sym^2 H^2(A) \oplus \left(\bigwedge^2 H^1(A) \otimes H^2(A)\right) \oplus \bigwedge^4 H^1(A).
\]
Decompose each summand (as $\Sp(4, \Q_\ell)$-representations) into a direct sum of irreducible representations:
\begin{enumerate}
  \item By Lemma \ref{lem:X2-local-systems},
  \[
    H^4(A; \Q_\ell) \cong \Q_\ell.
  \]

  \item By Lemmas \ref{lem:X2-local-systems} and \ref{lem:rep-thy-1,0} for the first and second isomorphisms respectively,
  \[
    H^1(A) \otimes H^3(A) \cong W_{1,0} \otimes W_{1,0} \cong \Q_\ell \oplus W_{2,0} \oplus W_{1,1}.
  \]

  \item By Lemma \ref{lem:X2-local-systems} and \cite[(B.2)]{fulton--harris} for the first and second isomorphisms respectively, 
  \begin{align*}
  \Sym^2 H^2(A) &\cong \Sym^2 (\Q_\ell \oplus W_{1,1}) \cong \bigoplus_{a = 0}^2 \Sym^a \Q_\ell \otimes \Sym^{2-a} W_{1,1} = \Sym^2 W_{1,1} \oplus W_{1,1} \oplus \Q_\ell.
  \end{align*}
  By \cite[Exercise 16.11]{fulton--harris}, 
  \[
    \Sym^2 W_{1,1} \oplus W_{1,1} \oplus \Q_\ell \cong \left( W_{2,2} \oplus W_{0,0} \right) \oplus W_{1,1} \oplus \Q_\ell \cong W_{2,2} \oplus W_{1,1} \oplus 2\Q_\ell.
  \]

  \item By Lemma \ref{lem:X2-local-systems} and distributivity of tensor products over direct sums,
  \[
    \bigwedge^2 H^1(A) \otimes H^2(A) \cong (\Q_\ell \oplus W_{1,1}) \otimes (\Q_\ell \oplus W_{1,1}) \cong \Q_\ell \oplus 2 W_{1,1} \oplus (W_{1,1} \otimes W_{1,1}). 
  \]
  By Lemma \ref{lem:rep-thy-1,1}, 
  \[
    \bigwedge^2 H^1(A) \otimes H^2(A) \cong \Q_\ell \oplus 2 W_{1,1} \oplus (\Q_\ell \oplus W_{2,0} \oplus W_{2,2}) = 2\Q_\ell \oplus 2W_{1,1} \oplus W_{2,0} \oplus W_{2,2}.
  \]

  \item By Lemma \ref{lem:X2-local-systems},
  \[
    \bigwedge^4 H^1(A) \cong \Q_\ell.
  \]
\end{enumerate}

Collect all the terms above to see that as $\Sp(4, \Q_\ell)$-representations,
\[
  H^4(\Sym^nA; \Q_\ell) \cong 7\Q_\ell \oplus 4W_{1,1} \oplus 2 W_{2,0}  \oplus 2W_{2,2}
\]
as claimed. Now add in Tate twists to the local systems corresponding to the appropriate $\GSp(4,\Q_\ell)$-representations as in the proof of Lemma \ref{lem:X2-local-systems}.
\end{proof}

Proposition \ref{prop:example-sym-n} provides the inputs to the computation of the cohomology of $\mathcal X_2^{\Sym(n)}$ in the same way as in Sections \ref{sec:universal-surface} and \ref{sec:fiber-powers}.
\begin{thm:low-degree-sym-n}
For all $n \geq k$ for $k$ even and for all $n \geq k-1$ for $k$ odd,
\[
  H^k(\mathcal X_2^{\Sym(n)}; \Q_\ell) = \begin{cases}
  \Q_\ell & k = 0 \\
  0 & k = 1, 3 \\
  3\Q_\ell(-1) & k = 2 \\
  9\Q_\ell(-2) & k = 4 \\
  2\Q_\ell(-5) & k = 5
  \end{cases}
\]
up to semi-simplification.
\end{thm:low-degree-sym-n}
\begin{proof}
Denote the $(p, q)$-entry on the $E_2$-sheet of the Leray spectral sequence (given by Proposition \ref{prop:leray}) of $\pi^n: \mathcal X_2^{\Sym(n)} \to \mathcal A_2$ by $E_2^{p, q}(n) = H^p(\mathcal A_2; H^q(\Sym^n A))$. This spectral sequence degenerates on the $E_2$-page. Applying Proposition \ref{prop:example-sym-n} and Corollary \ref{cor:petersen} yields $E_2^{p, q}(n)$ for $n \geq q$ and $q = 0$, $2$, $4$, which we record in Figure \ref{fig:Sym-n-spectral-sequence}. Observe also that $E_2^{p, q}(n) = 0$ for all $n \geq 0$ if $q$ is odd or if $p > 4$ by Theorem \ref{thm:petersen}. The theorem now follows directly.
\end{proof}

\begin{figure}
\centering
\includegraphics[width=\textwidth]{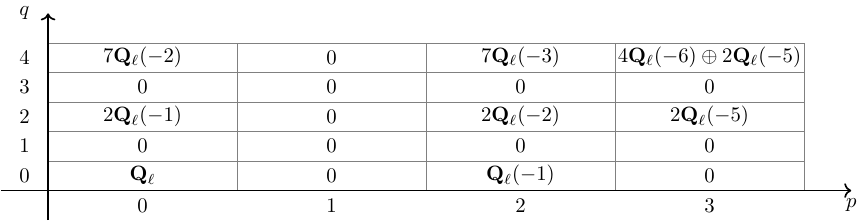}
\caption{Some nonzero terms $E_2^{p, q}(n)$ of the Leray spectral sequence for $\pi^n: \mathcal X_2^{\Sym(n)} \to \mathcal A_2$, for $n \geq q$ for $q$ even and for all $n \geq 0$ for $q$ odd. Note that $E_2^{1,q} = 0$ for all $q \geq 0$ and $E_2^{p, q} = 0$ for all $p \geq 4$ and $q = 0, 1, 2, 3$.}
\label{fig:Sym-n-spectral-sequence}
\end{figure}

\subsection{Explicit computations for $n = 2$.}
We compute $H^*(\mathcal X_2^{\Sym(2)}; \Q_\ell)$ completely. We first need the following.

\begin{lem}\label{lem:Sym2-local-systems}
There are isomorphisms of local systems
\begin{align*}
H^k(\Sym^2A; \Q_\ell) &\cong \begin{cases}
\Q_\ell & k = 0 \\
\mathbf V_{1,0} & k = 1 \\
2\Q_\ell(-1) \oplus 2\mathbf V_{1,1} &k = 2\\
3\mathbf V_{1,0}(-1) \oplus \mathbf V_{2, 1} & k = 3\\
4\Q_\ell(-2) \oplus 2 \mathbf V_{1,1}(-1) \oplus \mathbf V_{2,0}(-1) \oplus \mathbf V_{2, 2} & k = 4 \\
3\mathbf V_{1,0}(-2) \oplus \mathbf V_{2, 1}(-1) & k = 5\\
2\Q_\ell(-3) \oplus 2\mathbf V_{1,1}(-2) &k = 6\\
\mathbf V_{1,0}(-3) & k = 7 \\
\Q_\ell(-4) & k = 8 \\
0 & k > 8.
\end{cases}
\end{align*}
\end{lem}
\begin{proof}
This is a direct computation using Lemma \ref{lem:sn-fixed}.
\end{proof}

As usual, we want to compute $H^p(\mathcal A_2; H^q(\Sym^nA))$ for all $p$, $q \geq 0$. With Lemma \ref{lem:Sym2-local-systems}, this process is completely analogous to that of Sections \ref{sec:universal-surface} and \ref{sec:fiber-powers}. Therefore, we list the results below and omit the explanations. 
\begin{thm}\label{thm:sym-2}
The cohomology of $\mathcal X_2^{\Sym(2)}$ is given by
\[
  H^k(\mathcal X_2^{\Sym(2)}; \Q_\ell) = \begin{cases}
  \Q_\ell & k = 0 \\
  0 & k = 1, 3, k > 10 \\
  3\Q_\ell(-1) & k = 2 \\
  6\Q_\ell(-2) & k = 4 \\
  2\Q_\ell(-5) & k = 5 \\
  6\Q_\ell(-3) & k = 6 \\
  2\Q_\ell(-6) \oplus \Q_\ell(-5)& k = 7 \\
  3\Q_\ell(-4) & k = 8  \\
  2\Q_\ell(-7) & k = 9 \\
  \Q_\ell(-5) & k = 10
  \end{cases}
\]
up to semi-simplification.
\end{thm}
\begin{proof}
By Proposition \ref{prop:leray}, there is a spectral sequence with $E_2^{p, q} = H^p(\mathcal A_2; H^q(\Sym^2 A))$ which degenerates on the $E_2$-page and converges to $H^*(\mathcal X_2^{\Sym(2)}; \Q_\ell)$. Using Lemma \ref{lem:Sym2-local-systems} and Corollary \ref{cor:petersen}, we can compute $E_2^{p,q}$ for all $p, q \geq 0$. The results are recorded in Figure \ref{fig:Sym2-spectral-sequence}, from which the theorem follows directly.
\end{proof}

\begin{figure}
\centering
\includegraphics[width=\textwidth]{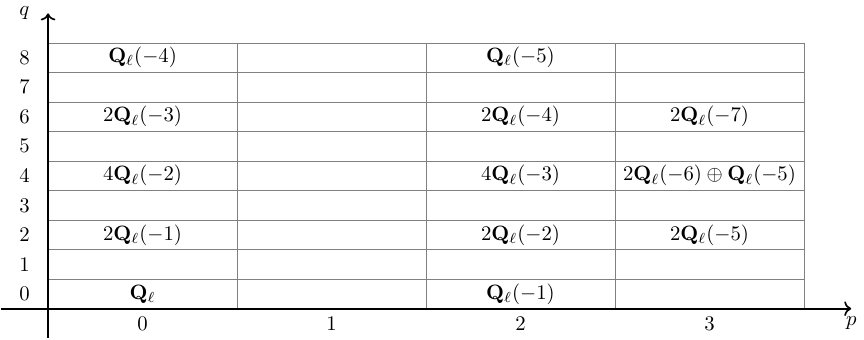}
\caption{The nonzero terms on the $E_2$-page of the Leray spectral sequence for $\pi^2: \mathcal X_2^{\Sym(2)} \to \mathcal A_2$.}
\label{fig:Sym2-spectral-sequence}
\end{figure}

\begin{rmk}
Note for all $n \geq 2$, the $E_2^{p, q}$ term in the spectral sequence for $\mathcal X_2^{\Sym(n)} \to \mathcal A_2$ for $q = 0, 1, 2, 3$ remain stable and are given in the corresponding entries in Figure \ref{fig:Sym2-spectral-sequence}.
\end{rmk}

\section{Arithmetic Statistics}\label{sec:arithmetic-statistics}
In this section, we apply the cohomological results of the previous sections to obtain arithmetic statistics results about abelian surfaces over finite fields. In Section \ref{sec:level-structures} we point out that the techniques of this paper can be applied to give arithmetic statistics about abelian surfaces with an ordered basis of its $N$-torsion given the cohomology of local systems of $\mathcal A_2[N]$, and apply the conjectural formulas (\cite{bergstrom-faber-vandergeer}) in the case $N = 2$ as an example.

Given the \'etale cohomology of a variety over a finite field $\F_q$, one can use the Grothendieck--Lefschetz trace formula to immediately deduce the number of $\F_q$-points on the variety. Even though the spaces $\mathcal X$ studied in this paper are not varieties but rather algebraic stacks, there is fortunately an applicable generalization, the Grothendieck--Lefschetz--Behrend trace formula, which gives the \emph{groupoid cardinality} of their $\F_q$-points.
\begin{defn}
Let $X$ be a groupoid. The \emph{groupoid cardinality} $\#X$ is defined as

\[
  \#X = \sum_{x \in X}\frac{1}{\#\Aut(x)}.
\]
\end{defn}

The following definition is necessary in order to state the Grothendieck--Lefschetz--Behrend trace formula.
\begin{defn}
Let $\mathcal X$ be a smooth Deligne--Mumford stack of finite type over $\F_q$. Let $\mathcal X_{\overline \F_q, \text{sm}}$ be the smooth site associated to $\mathcal X_{\overline \F_q}$. The arithmetic Frobenius acting on $H^i(\mathcal X_{\overline \F_q, \text{sm}}, \Q_\ell)$ is denoted by $\Phi_q: H^i(\mathcal X_{\overline \F_q, \text{sm}}, \Q_\ell) \to H^i(\mathcal X_{\overline \F_q,\text{sm}}, \Q_\ell)$. The action of $\Phi_q$ on $\Q_\ell(1)$ is multiplication by $q$. For any $n \in \Z$, the action of $\Phi_q$ on $\Q_\ell(n)$ is multiplication by $q^n$. 
\end{defn}

\begin{rmk}
For Deligne--Mumford stacks, the \'etale and smooth cohomology of abelian sheaves coincide: the \'etale and smooth cohomology of an abelian sheaf on schemes coincide by \cite[\href{https://stacks.math.columbia.edu/tag/03YY}{Lemma 03YY}]{stacks-project} and so the same holds for Deligne--Mumford stacks by \'etale descent since such stacks admit \'etale covers by schemes. All stacks in this section are Deligne--Mumford stacks over finite fields $\F_q$ (or their algebraic closures $\overline \F_q$) of any characteristic (see Section \ref{sec:spaces}). We will omit the distinction and just write $H^*(\mathcal X; \Q_\ell)$ for \'etale (or smooth) cohomology of $\mathcal X_{\overline \F_q}$.
\end{rmk}

The main tool of this section is the following trace formula.
\begin{thm}[Grothendieck--Lefschetz--Behrend trace formula, {\cite[Theorem 3.1.2]{behrend-inventiones}}]\label{thm:glb-trace}
Let $\mathcal X$ be a smooth Deligne--Mumford stack of finite type and constant dimension over the finite field $\F_q$. Then 
\[
  q^{\dim \mathcal X} \sum_{k \geq0} (-1)^k (\tr \Phi_q \mid H^k(\mathcal X_{\overline \F_q, \text{sm}}, \Q_\ell)) = \sum_{\xi \in[\mathcal X(\F_q)]} \frac{1}{\#\Aut(\xi)} =: \#\mathcal X(\F_q).
\]
\end{thm}

\subsection{Applying the Grothendieck--Lefschetz--Behrend trace formula.}\label{sec:glb-trace}
In this subsection, we apply the trace formula (Theorem \ref{thm:glb-trace}) to deduce corollaries of the cohomological results of the previous sections. Although all cohomology computations in the previous sections are only up to semi-simplification, the trace of a linear operator does not change under semi-simplification. Therefore, we apply the trace formula (Theorem \ref{thm:glb-trace}) to the semi-simplification without making this distinction.

Recall that the interpretation of the $\F_q$-points of each stack in question is given in Section \ref{sec:spaces}. This first count is well-known, but we list it below for completeness.
\begin{thm}
\[
  \#\mathcal A_2(\F_q) = q^3 + q^2. 
\]
\end{thm}
\begin{proof}
Using Theorem \ref{thm:a2}, apply the trace formula (Theorem \ref{thm:glb-trace}), noting that $\dim \mathcal A_2 = 3$ and that the values $\tr(\Phi_q \mid H^k(\mathcal A_2; \Q_\ell))$ are known. 
\end{proof}

The point counts in the rest of this section are new to the best of our knowledge. 
\begin{thm}\label{thm:n2-counts}
\begin{align*}
  \# \mathcal X_2(\F_q) &= q^5 + 2q^4 + 2q^3 + q^2 - 1,\\
  \# \mathcal X_2^2(\F_q) &= q^7 + 4q^{6} + 9q^{5} + 9q^{4} + 3q^{3} - 5q^{2} - 5q -3,\\
  \#\mathcal X_2^{\Sym(2)}(\F_q) &= q^7 + 3q^6 + 6q^5 + 6q^4 + 3q^3 - 2q^2 - 2q - 2.
\end{align*}
\end{thm}
\begin{proof}
In all cases, the counts follow from the trace formula (Theorem \ref{thm:glb-trace}). Applying Theorem \ref{thm:universal-surface} and the fact that $\dim \mathcal X_2 = 5$,
\[
  \#\mathcal X_2(\F_q) = q^5(1 + 2q^{-1} + 2q^{-2} - q^{-5} + q^{-3}) = q^5 + 2q^4 + 2q^3 + q^2 - 1.
\]
Applying Theorem \ref{thm:2-power-universal-surface} and the fact that $\dim \mathcal X_2^2 = 7$,
\begin{align*}
\#\mathcal X_2^2(\F_q) &= q^7\left(1 + 4q^{-1} + 9q^{-2} - (3q^{-5} + q^{-4}) + 9q^{-3} \right) \\
&\quad +q^7 \left(- (3q^{-5} + 4q^{-6}) + 4q^{-4}-(3q^{-7} + q^{-6}) + q^{-5}\right) \\
&= q^7 + 4q^{6} + 9q^{5} + 9q^{4} + 3q^{3} - 5q^{2} - 5q -3.
\end{align*}
Applying Theorem \ref{thm:sym-2} and the fact that $\dim \mathcal X_2^{\Sym(2)} = 7$,
\begin{align*}
\#\mathcal X_2^{\Sym(2)}(\F_q) &= q^7 \left(1 + 3q^{-1} + 6q^{-2} - 2q^{-5} + 6q^{-3} \right) \\
&\quad+ q^7\left(- (q^{-5} + 2q^{-6}) + 3q^{-4} - 2q^{-7} + q^{-5}\right) \\
&= q^7 + 3q^6 + 6q^5 + 6q^4 + 3q^3 - 2q^2 - 2q - 2 \qedhere
\end{align*}
\end{proof}

We can also piece together the partial information we have about $H^k(\mathcal X_2^n; \Q_\ell)$ and $H^k(\mathcal X_2^{\Sym(n)};\Q_\ell)$ to give an approximation of $\#\mathcal X_2^n(\F_q)$ and $\#\mathcal X_2^{\Sym(n)}(\F_q)$, for fixed $n$ and asymptotic in $q$. 
\begin{thm}
For all $n \geq 1$,
\[
  \# \mathcal X_2^n(\F_q) = q^{3 + 2n} + \left(\binom{n+1}{2}+1\right)q^{2 + 2n} + \left(\frac{n(n+1)(n^2+n+2)}{8} + \binom{n+1}{2}\right)q^{1 + 2n} + O(q^{2n}).
\]
For $n = 3$,
\[
  \# \mathcal X_2^{\Sym(3)}(\F_q) = q^9+ 3q^8 + O(q^7)
\]
and for all $n \geq 4$, 
\[
  \# \mathcal X_2^{\Sym(n)}(\F_q) = q^{3+2n} + 3q^{2 + 2n} + 9q^{1 + 2n} + O(q^{2n}).
\]
\end{thm}
\begin{proof}
By Theorem \ref{thm:weights}, for all $p \geq 0$ and $a \geq b \geq 0$,
\[
  \lvert\tr(\Phi_q|H^p(\mathcal A_2; \mathbf V_{a, b}))\rvert \leq \dim H^p(\mathcal A_2; \mathbf V_{a,b}) q^{-\frac{p+a+b}{2}}
\]
and so for any $N \geq 0$ such that $N - p \equiv a+b \pmod 2$,
\[
  \left\lvert\tr\left(\Phi_q \biggr| H^p(\mathcal A_2; \mathbf V_{a,b})\left(\frac{a + b - (N-p)}{2}\right)\right)\right\rvert \leq \dim H^p(\mathcal A_2; \mathbf V_{a,b}) q^{-N/2}.
\]
For any $N \geq 0$ and $\mathcal X = \mathcal X_2^n$ or $\mathcal X_2^{\Sym(n)}$, this estimate, the trace formula (Theorem \ref{thm:glb-trace}), the properties of the Leray spectral sequence for $\mathcal X \to \mathcal A_2$ (Proposition \ref{prop:leray}), and Proposition \ref{prop:mult-even-odd} imply
\[
  \#\mathcal X(\F_q) = q^{3+2n}\left(\sum_{0 \leq k \leq N} (-1)^k\tr\left(\Phi_q | H^k(\mathcal X; \Q_\ell)\right)\right) + O\left(q^{3 + 2n-\frac{N+1}{2}}\right).
\]
For $\mathcal X = \mathcal X_2^n$, applying Theorem \ref{thm:low-degree-computation} with $N = 5$ gives
\[
  \#\mathcal X_2^n(\F_q) = q^{3 + 2n} + \left(\binom{n+1}{2}+1\right)q^{2 + 2n} + \left(\frac{n(n+1)(n^2+n+2)}{8} + \binom{n+1}{2}\right)q^{1 + 2n} + O(q^{2n})
\]
and for $\mathcal X = \mathcal X_2^{\Sym(n)}$, the same computation using Theorem \ref{thm:low-degree-sym-n} with $n = N = 3$ gives
\[
  \# \mathcal X_2^{\Sym(3)}(\F_q) = q^9+ 3q^8 + O(q^7)
\]
and with $n \geq 4$, $N = 5$ gives
\[
  \#\mathcal X_2^{\Sym(n)}(\F_q) = q^{3+2n} + 3q^{2 + 2n} + 9q^{1 + 2n} + O(q^{2n}). 
\]
Finally, we note that it is possible to compute the exact value of $\# \mathcal X_2^{\Sym(3)}$ analogously to the calculation of $\# \mathcal X_2^{\Sym(2)}$ in Theorem \ref{thm:n2-counts} but we omit it for brevity.
\end{proof}

These point counts imply the arithmetic statistics results outlined in Section \ref{sec:introduction} which we discuss for the remainder of this subsection.
\begin{lem}\label{lem:prob-defns}
Define a probability measure $\mathbf P$ on $[\mathcal A_2(\F_q)]$ by 
\[
  \mathbf P([A_0]) = \frac{1}{\#\mathcal A_2(\F_q)\#\Aut_{\F_q}(A_0)}
\]
for each $\F_q$-isomorphism class $[A_0]\in[\mathcal A_2(\F_q)]$. For fixed $n \geq 1$, the expected value of the number of $\F_q$-points on $n$th powers of abelian surfaces with respect to this probability measure is
\[
  \mathbf E[\#A^n(\F_q)] = \frac{\sum_{(A_0, p) \in [\mathcal X_2^n(\F_q)]} \frac{1}{\#\Aut_{\F_q}(A_0, p)}}{\sum_{A_0 \in [\mathcal A_2(\F_q)]} \frac{1}{\#\Aut_{\F_q}(A_0)}} = \frac{\#\mathcal X_2^n(\F_q)}{\#\mathcal A_2(\F_q)}.
\]
Similarly, the expected value of the groupoid cardinality of $\Sym^n A(\F_q)$ is
\[
  \mathbf E[\#\Sym^n A(\F_q)] = \frac{\sum_{(\Sym^n A_0, p) \in [\mathcal X_2^{\Sym(n)}(\F_q)]} \frac{1}{\#\Aut_{\F_q}(\Sym^n A_0, p)} }{\sum_{A_0 \in [\mathcal A_2(\F_q)]} \frac{1}{\#\Aut_{\F_q}(A_0)}} = \frac{\#\mathcal X_2^{\Sym(n)}(\F_q)}{\#\mathcal A_2(\F_q)}.
\]
\end{lem}
\begin{proof}
Consider a representative abelian surface $A_0$ in a fixed $\F_q$-isomorphism class $[A_0]$ and let $Z_0 = A_0^n$ or $\Sym^n A_0$. There is an action of $\Aut_{\F_q}(A_0)$ on $Z_0(\F_q)$. For any $p_0 \in Z_0(\F_q)$, its $\F_q$-isomorphism class is precisely its orbit under the action of $\Aut_{\F_q}(A_0)$. Let $\Stab_G(p_0)$ denote the stabilizer of $p_0$ in the group $G$. Then the automorphism group of the pair $(A_0^n, p_0)$ is $\Stab_{\Aut_{\F_q}(A_0)}(p_0)$ and the automorphism group of $(\Sym^n A_0, p_0)$ is $\Stab_{\Aut_{\F_q}(A_0)}(p_0) \times \Stab_{S_n}(p_0)$; this is a direct product because the action of $S_n$ and $\Aut_{\F_q}(A_0)$ on $A_0^n$ commute. 

Let $N(p_0) = \#\Stab_{S_n}(p_0)$ if $Z_0 = \Sym^n A_0$ and $N(p_0) = 1$ if $Z_0 = A_0^n$. Let $\Orb(p_0)$ denote the orbit of $p_0$ in $Z_0(\F_q)$ under the action of $\Aut_{\F_q}(A_0)$. The contribution of $A_0$ (and its corresponding fiber $Z_0$) to the expected value is
\[
  \frac{1}{\#\mathcal A_2(\F_q)}\sum_{p_0 \in [Z_0(\F_q)]} \frac{1}{\#\Aut_{\F_q}(Z_0, p_0)} = \frac{1}{\#\mathcal A_2(\F_q)}\sum_{p_0 \in [Z_0(\F_q)]} \frac{\#\Orb(p_0)}{N(p_0)\#\Aut_{\F_q}(A_0)} = \frac{\#Z_0(\F_q)}{\#\mathcal A_2(\F_q)\#\Aut_{\F_q}(A_0)}
\]
where the first equality follows from the orbit-stabilizer theorem and the second follows from the fact that the groupoid cardinality of $\Sym^n A_0(\F_q)$ is $\sum_{p_0 \in \Sym^nA_0(\F_q)} \frac{1}{N(p_0)}$.
\end{proof}

Lemma \ref{lem:prob-defns} and the results of this section immediately imply the statistics given in Section \ref{sec:introduction}; we restate them here for convenience.
\begin{cor:avg-fq-points}
The expected number of $\F_q$-points on abelian surfaces defined over $\F_q$ is
\[
  \mathbf E[\#A(\F_q)] = \frac{\#\mathcal X_2(\F_q)}{\# \mathcal A_2(\F_q)} = q^2 + q + 1 - \frac{1}{q^3 + q^2}. 
\]
\end{cor:avg-fq-points}

Because $\#A^n(\F_q) = \#A(\F_q)^n$ for all abelian surfaces $A$, the following corollary gives asymptotics for all moments of $\#A(\F_q)$ as well as the exact second moment.
\begin{cor:nth-fiber-asymptotic-in-q}
The expected value of $\#A^2(\F_q)$ is
\[
  \mathbf E[\#A^2(\F_q)] = \frac{\# \mathcal X_2^2(\F_q)}{\# \mathcal A_2(\F_q)}= q^4 + 3q^3 + 6q^2 + 3q - \frac{5q^2 + 5q + 3}{q^3 + q^2}
\]
and for all $n \geq 1$, 
\[
  \mathbf E[\# A^n(\F_q)] = \frac{\# \mathcal X_2^n(\F_q)}{\# \mathcal A_2(\F_q)}= q^{2n} + \binom{n+1}{2} q^{2n-1} + \left(\frac{n(n+1)(n^2+n+2)}{8}\right) q^{2n-2} + O(q^{2n-3}).
\]
\end{cor:nth-fiber-asymptotic-in-q}
Recall that $\Sym^n A(\F_q)$ for an abelian surface $A$ and any $n \geq 1$ is the set of $n$-tuples defined over $\F_q$ as tuples, i.e. the $n$ points are permuted by $\Frob_q$.
\begin{cor:sym-asymptotic-in-q}
The expected value of $\#\Sym^n A(\F_q)$ for $n = 2$ is 
\[
  \mathbf E[\#\Sym^2 A(\F_q)] = \frac{\#\mathcal X_2^{\Sym(2)}(\F_q)}{\#\mathcal A_2(\F_q)}= q^4 + 2q^3 + 4q^2 + 2q + 1 - \frac{3q^2 + 2q + 2}{q^3 + q^2}.
\]
For $n = 3$,
\[
  \mathbf E[\#\Sym^3 A(\F_q)] = \frac{\#\mathcal X_2^{\Sym(3)}(\F_q)}{\#\mathcal A_2(\F_q)} = q^6 + 2q^5 + O(q^4)
\]
and for all $n \geq 4$,
\[
  \mathbf E[\#\Sym^nA(\F_q)] =\frac{\#\mathcal X_2^{\Sym(n)}(\F_q)}{\#\mathcal A_2(\F_q)}= q^{2n} + 2q^{2n-1} + 7q^{2n-2} + O(q^{2n-3}).
\]
\end{cor:sym-asymptotic-in-q}

The fact that we have determined the exact value for the second moment means we can calculate the variance of $\#A(\F_q)$ using Corollary \ref{cor:avg-fq-points}.
\begin{cor:variance}
The variance of $\#A(\F_q)$ is
\[
  \Var(\#A(\F_q)) = \mathbf E[\#A^2(\F_q)] - (\mathbf E[\#A(\F_q)])^2 = q^3 + 3q^2 + q - 1 - \frac{3q^2 + 3q + 1}{q^3+ q^2} - \frac{1}{(q^3 + q^2)^2}.
\]
\end{cor:variance}

\subsection{Level structures.}\label{sec:level-structures}
Let $N \geq 2$ and let $\pi_N: \mathcal X_2[N] \to \mathcal A_2[N]$ be the projection map. For each local system $\mathbf V_{a,b}$ on $\mathcal A_2$, we can define local systems on $\mathcal A_2[N]$ (also denoted $\mathbf V_{a,b}$) via pullback by the map $pr: \mathcal A_2[N] \to \mathcal A_2$. By \cite[\href{https://stacks.math.columbia.edu/tag/075H}{Lemma 075H}]{stacks-project}, there is an isomorphism of sheaves
\[
  pr^*(R^q\pi_*\Q_\ell) \cong R^q (\pi_N)_*(pr^*\Q_\ell) \cong R^q(\pi_N)_* \Q_\ell.
\]
Therefore, the decomposition given in Lemma \ref{lem:X2-local-systems} also holds for the local systems $H^k(A; \Q_\ell)$ on $\mathcal A_2[N]$, interpreting all local systems $\mathbf V_{a,b}$ as their pullbacks to $\mathcal A_2[N]$.

However, the cohomology of local systems on $\mathcal A_2[N]$ is not yet known in general. In the case $N = 2$, many parts of the Euler characteristics of local systems $\mathbf V_{a,b}$ on $\mathcal A_2[2]$ are known and there are conjectures for the rest; this is done in \cite{bergstrom-faber-vandergeer}. Here, we consider the compactly supported Euler characteristics of such local systems, defined
\[
  e_c(\mathcal A_2[2]; \mathbf V_{a, b}) := \sum_{k \geq 0} (-1)^k [H^k_c(\mathcal A_2[2]; \mathbf V_{a, b})]
\]
taken in the Grothendieck group of an appropriate category, e.g. the category of mixed Hodge structures or of Galois representations.

\begin{conj}[Bergstr\"om--Faber--van der Geer, {\cite[Section 10]{bergstrom-faber-vandergeer}}]
The compactly supported Euler characteristic of $\mathbf V_{1,1}$ over $\mathcal A_2[2]$ is given by $5\Q_\ell(-3) - 10 \Q_\ell(-2)$. The compactly supported Euler characteristic of $\mathbf V_{0,0}$ over $\mathcal A_2[2]$ is given by $\Q_\ell(-3) + \Q_\ell(-2) - 14 \Q_\ell(-1) + 16\Q_\ell$.
\end{conj}
The cohomology of $\mathcal A_2[2]$ is also computed in \cite[Theorem 5.2.1]{lee--weintraub}. Assuming these two calculations and using the Leray spectral sequence of $\mathcal X_2[2] \to \mathcal A_2[2]$,
\begin{align*}
  e_c(\mathcal X_2[2]; \Q_\ell) &= \sum_{k \geq 0}(-1)^k H^k_c(\mathcal X_2[2]; \Q_\ell) = \sum_{p, q \geq 0} (-1)^{p+q} H_c^p(\mathcal A_2[2]; H^q(A; \Q_\ell))\\
  &= e_c(\mathcal A_2[2]; \mathbf V_{0,0}) + e_c(\mathcal A_2[2]; \mathbf V_{0,0}(-1)) + e_c(\mathcal A_2[2]; \mathbf V_{0,0}(-2)) + e_c(\mathcal A_2[2]; \mathbf V_{1,1}) \\
  &= \Q_\ell(-5) + 2\Q_\ell(-4) - 7 \Q_\ell(-3) - 7 \Q_\ell(-2) + 2\Q_\ell(-1) + 16 \Q_\ell.
\end{align*}
These computations plus a version of the trace formula (Theorem \ref{thm:glb-trace}) for compactly supported cohomology imply that
\begin{align*}
\#\mathcal A_2[2](\F_q) &= q^3 + q^2 - 14 q + 16, \\
\#\mathcal X_2[2](\F_q) &= q^5 + 2q^4 - 7q^3 - 7q^2 + 2q + 16.
\end{align*}
In particular, note that an $\F_q$-point on $\mathcal A_2[N]$ corresponds to an abelian surface $A$ defined over $\F_q$ with an ordered basis of its $N$-torsion defined over $\F_q$. Therefore, careful analysis of the counts $\#\mathcal A_2[2](\F_q)$ and $\#\mathcal X_2[2](\F_q)$ will yield the average number of abelian surfaces over $\F_q$ with $2$-torsion defined over $\F_q$ and the average number of $\F_q$-points on such abelian surfaces. 

\appendix

\section{Tensor Products of Irreducible $\Sp(4)$-Representations}\label{sec:tensors}

In this appendix we summarize the combinatorial results decomposing tensor products of irreducible $\Sp(4)$-representations into irreducible ones and prove Lemmas \ref{lem:rep-thy-1,0} and \ref{lem:rep-thy-1,1}. As usual, we let $W_{a,b}$ denote the irreducible $\Sp(4)$-representation corresponding to the partition $a \geq b \geq 0$.

Let $\lambda=(\lambda_1, \dots, \lambda_n)$ be a partition, so that $\lambda_1 \geq \dots \geq \lambda_n \geq 0$. A \emph{Young diagram of shape $\lambda$} is an arrangement of left-justified rows of boxes, such that row $i$ has $\lambda_i$-many boxes. A \emph{skew shape $\lambda/\mu$}, where $\lambda$ and $\mu$ are both partitions with $\mu\subseteq \lambda$ is the arrangement of rows of boxes given by the Young diagram of shape $\lambda$, with the Young diagram of shape $\mu$ erased. A \emph{skew tableau of shape $\lambda/\mu$ with content $\beta = (\beta_1, \dots, \beta_k)$} is a labeling of a skew shape $\lambda/\mu$ where $\beta_i$-many of the boxes are labeled with the number $i$. Such a tableau is called \emph{semi-standard} if the labels are nondecreasing along rows and increasing along columns. Given a skew semi-standard tableau of shape $\lambda/\mu$ with content $\beta$, we may demand that the concatenation of the reversed rows is a \emph{lattice word}: list the entries of the tableau from right to left, starting from the top row and working down; for any $t$ smaller than the length of this list, the first $t$ elements must contain as many entries $i$ as it contains entries $i+1$. For example, consider the tableau on the left in Figure \ref{fig:LR-example}; the concatenation of the reversed rows is ``1121.'' For $t = 3$, the list of the first $3$ entries is ``112,'' and there are more entries labeled ``1'' than there are ``2'' in ``112'' (and of course, more entries labeled ``2'' than there are ``3,'' and so on). This is true for all $t \leq 4$ for this concatenation of the reversed rows for this tableau.

Skew semi-standard tableaux whose concatenation of the reversed rows is a lattice word are called \emph{Littlewood--Richardson tableaux}. For another description of these tableaux, see \cite[p. 456]{fulton--harris}. As an example, we give all Littlewood--Richardson tableaux of shape $(4, 2, 1)/(2, 1)$ and content $(3,1)$ in Figure \ref{fig:LR-example}. Finally, the \emph{Littlewood--Richardson coefficient $c_{\alpha\beta}^\gamma$} is the number of Littlewood--Richardson tableaux of shape $\gamma/\alpha$ and content $\beta$. One important property about Littlewood--Richardson coefficients is that $c_{\alpha\beta}^\gamma = c_{\beta\alpha}^\gamma$ for all partitions $\alpha, \beta, \gamma$. One can see this by applying the Littlewood--Richardson rule (\cite[(15.23), (A.8)]{fulton--harris}, \cite[Theorem 1.4.4]{koike--terada}) which says that $c_{\alpha\beta}^\gamma$ is the multiplicity of $\mathbf S_\gamma(V)$ in $\mathbf S_\alpha(V)\otimes \mathbf S_\beta(V)$ where $\mathbf S_\gamma(V)$, $\mathbf S_\alpha(V)$, and $\mathbf S_\beta(V)$ are the irreducible representations of $\GL(n)$ corresponding to the partitions $\alpha, \beta, \gamma$ using the notation of \cite[Section 15.3]{fulton--harris}.

\begin{figure}
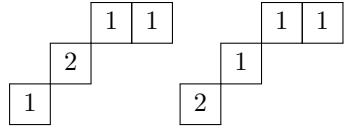

\begin{ytableau}
\none & \none & 1 & 1 \\
\none & 2 \\
1
\end{ytableau} 
\begin{ytableau}
\none & \none & 1 & 1 \\
\none & 1 \\
2
\end{ytableau}
\caption{All Littlewood--Richardson tableaux of shape $(4, 2, 1)/(2, 1)$ and content $(3,1)$, showing that $c^{(4,2,1)}_{(2,1)(3,1)} = 2$.}
\label{fig:LR-example}
\end{figure}

Let $\mathscr P$ be the set of all partitions. There is a \emph{universal character ring} $\Lambda$ (a $\Z$-algebra defined in \cite[Section 1.4]{koike--terada}) with a $\Z$-basis $\{\chi_{\Sp}(\lambda)\}_{\lambda \in \mathscr P}$ (\cite[Definition 2.1.1, Proposition 2.1.2]{koike--terada}). The structure constants of $\Lambda$ with respect to the $\Z$-basis $\{\chi_{\Sp}(\lambda)\}_{\lambda \in \mathscr P}$ are given by \emph{Newell--Littlewood numbers}:
\begin{thm}[{\cite[Theorem 3.1]{koike}}]\label{thm:tensors}
For any $\mu, \nu \in \mathscr P$,
\[
  \chi_{\Sp}(\mu) \chi_{\Sp}(\nu) = \sum_{\lambda \in \mathscr P} N_{\mu\nu\lambda} \chi_{\Sp}(\lambda)
\]
with $N_{\mu\nu\lambda} = \sum_{\zeta, \sigma, \tau} c_{\zeta\sigma}^\mu c_{\zeta\tau}^\nu c_{\sigma\tau}^\lambda$. Here, $c_{\alpha\beta}^\gamma$ is a Littlewood--Richardson coefficient, i.e. the number of Littlewood--Richardson tableaux of shape $\gamma/\alpha$ and content $\beta$. The constants $N_{\mu\nu\lambda}$ are known as \emph{Newell--Littlewood numbers}.
\end{thm}
There is an algebra homomorphism $\pi_{\Sp(2n)}: \Lambda \to R(\Sp(2n))$ where $R(\Sp(2n))$ is the character ring of $\Sp(2n,\Q_\ell)$ called the \emph{specialization homomorphism} (\cite[Section 2.2]{koike--terada}). If $\lambda \in \mathscr P$ with $\ell(\lambda) \leq n$ where $\ell(\lambda)$ is the length of the partition $\lambda$, then $\pi_{\Sp(2n)}(\chi_{\Sp}(\lambda)) = \chi_{\Sp(2n)}(\lambda)$, the character of the irreducible $\Sp(2n)$-representation corresponding to $\lambda$ (\cite[Proposition 2.2.1]{koike--terada}). The images $\pi_{\Sp(2n)}(\chi_{\Sp}(\lambda)) \in R(\Sp(2n))$ for $\lambda \in \mathscr P$ such that $\ell(\lambda) > n$ are computed in \cite[Section 2.4]{koike--terada} and outlined below:

If $\ell(\lambda) > n$, let $\lambda_i'$ denote the number of boxes in the $i$th column of the Young diagram of $\lambda$ for all $i \geq 1$ and let $\ell$ be the total number of columns. If there exists $i \geq 1$ for which $\lambda_i' - (i-1) = n+1$, then $\pi_{\Sp(2n)}(\chi_{\Sp}(\lambda)) = 0$. 

Now assume $\lambda_i' - (i-1) \neq n+1$ for all $1 \leq i \leq \ell$. For $1 \leq i \leq \ell$, define
\[
  k_i = \begin{cases}
  \lambda_i' & \lambda_i' - (i-1) \leq n, \\
  2n + 2i - \lambda_i' & \lambda_i' - (i-1)> n+1
  \end{cases}
\]
and define $t_i = k_i - (i-1)$. Under these assumptions, $n \geq t_i$ for all $1 \leq i \leq \ell$. If $t_i = t_j$ for some $1 \leq i < j \leq \ell$, then $\pi_{\Sp(2n)}(\chi_{\Sp}(\lambda)) = 0$. Otherwise, reorder the numbers $t_i$ in decreasing order with 
\[
  n \geq t_{i_1} > t_{i_2} > \dots > t_{i_\ell}.
\]
Define $\mu_k' = t_{i_k} + (k-1)$ for $1 \leq k \leq \ell$. If $\mu_\ell' < 0$ then $\pi_{\Sp(2n)}(\chi_{\Sp}(\lambda)) = 0$. Otherwise, let $\mu$ be the Young diagram for which $\mu_k'$ is the number of boxes in the $k$th column. Then $\ell(\mu) \leq n$ and 
\[
  \pi_{\Sp(2n)}(\chi_{\Sp}(\lambda)) = (-1)^{\textup{sgn}(\sigma) + s} \chi_{\Sp(2n)}(\mu)
\]
where $\sigma$ is the permutation with $\sigma(j) = i_j$ for all $1 \leq j \leq \ell$ and $s$ is the number of indices $1 \leq i \leq \ell$ such that $\lambda_i' - (i-1) > n$.

With the arithmetic of $\Lambda$ given in Theorem \ref{thm:tensors} and the specialization homomorphism $\pi_{\Sp(2n)}$, we are ready to prove Lemmas \ref{lem:rep-thy-1,0} and \ref{lem:rep-thy-1,1}. Similarly as in Section \ref{sec:fiber-powers}, we set $\chi_{\Sp}(\lambda_1, \dots, \lambda_n) := 0$ and $\chi_{\Sp(2n)}(\lambda_1, \dots, \lambda_n) := 0$ if $(\lambda_1, \dots, \lambda_n)$ is not a partition, or more specifically if $\lambda_i > \lambda_{i-1}$ or $\lambda_i < 0$ for some $i$.
\begin{proof}[Proof of Lemma \ref{lem:rep-thy-1,0}]
By Theorem \ref{thm:tensors}, 
\[
  \chi_{\Sp}(1) \chi_{\Sp}(a, b) = \sum_{\lambda \in \mathscr P} N_{(1)(a,b)\lambda}\chi_{\Sp}(\lambda)
\]
with $N_{(1)(a,b)\lambda} = \sum_{\zeta, \sigma, \tau} c_{\zeta\sigma}^{(1)} c_{\zeta\tau}^{(a, b)} c_{\sigma\tau}^\lambda$. According to \cite[p. 509 (1)]{koike--terada}, the above sum is given by 
\[
  \chi_{\Sp}(1) \chi_{\Sp}(a, b) = \chi_{\Sp}(a-1, b) + \chi_{\Sp}(a+1, b) + \chi_{\Sp}(a, b-1) + \chi_{\Sp}(a, b+1) + \chi_{\Sp}(a, b, 1).
\]
Applying the specialization homomorphism with \cite[Proposition 2.2.1]{koike--terada} gives
\begin{align*}
  \chi_{\Sp(4)}(1, 0) \chi_{\Sp(4)}(a,b) = & \, \chi_{\Sp(4)}(a-1, b) + \chi_{\Sp(4)}(a+1, b) + \chi_{\Sp(4)}(a, b-1) \\
  &\quad+ \chi_{\Sp(4)}(a, b+1) + \pi_{\Sp(4)}(\chi_{\Sp}(a, b, 1)). \tag{$*$}
\end{align*}
Now we determine $\pi_{\Sp(4)}(\chi_{\Sp}(a,b,1))$. Compute that 
\[
  \lambda_i' = \begin{cases}
  3 & i = 1, \\
  2 & 2 \leq i \leq b,\\
  1 & b < i \leq a.
  \end{cases}
\]
With $i = 1$ and $2n = 4$, 
\[
  \lambda_i' - (i-1) = 3 = n+1
\]
which implies that $\pi_{\Sp(4)}(\chi_{\Sp}(a,b,1)) = 0$.

The character of the $\Sp(4, \Q_\ell)$-representation $W_{1,0}\otimes W_{a,b}$ is the product $\chi_{\Sp(4)}(1) \chi_{\Sp(4)}(a,b)$. Therefore, rewriting $(*)$ on the level of $\Sp(4, \Q_\ell)$-representations proves this lemma.
\end{proof}

\begin{proof}[Proof of Lemma \ref{lem:rep-thy-1,1}]
Let $a \geq b \geq 0$. Suppose $\zeta$, $\sigma$, $\tau$, and $\lambda$ are partitions such that $c_{\zeta\sigma}^{(1,1)}c_{\zeta\tau}^{(a, b)}c_{\sigma\tau}^\lambda \neq 0$; we first list all possibilities for such partitions $\zeta$, $\sigma$, $\tau$, and $\lambda$. Throughout this proof, we define $c_{\alpha\beta}^\gamma := 0$ and $N_{\alpha\beta\gamma} := 0$ if any of $\alpha$, $\beta$, or $\gamma$ are tuples of nonnegative integers but are not partitions.

Since $\zeta, \sigma \subseteq (1,1)$, $\zeta, \tau \subseteq (a, b)$, we must have $\zeta = (\zeta_1, \zeta_2)$, $\sigma = (\sigma_1, \sigma_2)$, and $\tau = (\tau_1, \tau_2)$. 

Consider $c_{\zeta\sigma}^{(1,1)}$. We have $\zeta, \sigma \subseteq (1,1)$ with $\zeta_1 + \zeta_2 + \sigma_1 + \sigma_2 = 2$. This forces the three possibilities: (1) $\zeta = (1,1)$ and $\sigma = (0,0)$, (2) $\zeta = (1,0)$ and $\sigma = (1,0)$, or (3) $\zeta = (0,0)$ and $\sigma = (1,1)$. 

Next, consider $c_{\zeta\tau}^{(a, b)} = c_{\tau\zeta}^{(a, b)}$, and the above three possibilities.
\begin{enumerate}
  \item If $\zeta = (1,1)$, then a tableau with shape $(a, b)/\zeta$ has $a-1$ boxes in the first row and $b-1$ boxes in the second row. In order for the tableau to satisfy the lattice word condition and have the labels be increasing within each row, the entire first row must be labeled $1$. Therefore, the content $(\tau_1, \tau_2)$ must satisfy $\tau_1 \geq a-1$ with $\tau_1 + \tau_2 = a + b - 2$. Because $\tau \subseteq (a, b)$, this forces two possibilities: $\tau = (a-1, b-1)$ or $\tau = (a, b-2)$. 

  Suppose $\tau = (a, b-2)$. We count tableau with shape $(a, b)/\tau$ and content $\zeta = (1,1)$. The tableau of shape $(a, b)/(a, b-2)$ has one row with two boxes, and the lattice word condition imposes that all boxes of the first row must be labeled $1$. Therefore, there are no such tableau with content $\zeta = (1,1)$. For $\tau = (a-1, b-1)$, see Figure \ref{fig:LR-coeff} for the unique tableau with shape $(a, b)/(a-1, b-1)$ and content $\zeta = (1,1)$. 

  \item If $\zeta = (1,0)$, then a tableau with shape $(a, b)/\tau$ and content $\zeta = (1,0)$ must satisfy $\tau = (a-1, b)$ or $(a, b-1)$. In both cases, such a tableau is the unique tableau with one box.

  \item If $\zeta = (0,0)$, then a tableau with shape $(a, b)/\tau$ and content $\zeta = (0,0)$ must satisfy $\tau = (a,b)$. In this case, such a tableau must be the empty one.
\end{enumerate}

Lastly, consider $c_{\sigma\tau}^\lambda = c_{\tau\sigma}^\lambda$. 
\begin{enumerate}
  \item If $\zeta = (1,1)$, $\sigma = (0,0)$, and $\tau = (a-1, b-1)$, then a tableau with shape $\lambda/\tau$ and content $\sigma = (0,0)$ must satisfy $\lambda = (a-1, b-1)$. The only such tableau is the empty one.

  \item If $\zeta = (1,0)$, $\sigma = (1,0)$, and $\tau = (a-1, b)$ or $(a, b-1)$, then a tableau with shape $\lambda/\tau$ and content $\sigma = (1,0)$ must satisfy $\sum_{i\geq 1}\lambda_i = a + b$ with $\tau \subseteq \lambda = (\lambda_1, \lambda_2, \dots)$. If $\tau = (a-1, b)$, then $\lambda = (a, b)$, $(a-1, b+1)$, or $(a-1, b, 1)$. If $\tau = (a, b-1)$, then $\lambda = (a+1, b-1)$, $(a, b)$, or $(a, b-1, 1)$. In all cases, such a tableau must be the unique one with one box.

  \item If $\zeta = (0,0)$, $\sigma = (1,1)$, and $\tau = (a, b)$, then a tableau with shape $\lambda/\tau$ and content $\sigma = (1,1)$ must satisfy $\sum_{i\geq 1} \lambda_i = a+b+2$ with $\tau = (a,b) \subseteq \lambda = (\lambda_1, \lambda_2, \dots)$. Then $\lambda$ is one of
  \[
    (a+2, b),\, (a+1, b+1),\, (a, b+2),\, (a+1, b, 1),\, (a, b+1, 1), \, (a, b, 2), \, (a, b, 1, 1).
  \]
  \begin{enumerate}
    \item Suppose $\lambda = (a+2, b)$ or $(a, b+2)$. The tableau of shape $\lambda/(a,b)$ has one row with two boxes, and the lattice word condition imposes that all boxes of the first row must be labeled $1$. Therefore, there are no such tableau with content $\sigma = (1,1)$. 
    \item Suppose $\lambda = (a+1, b+1)$. See Figure \ref{fig:LR-coeff} for the unique tableau of shape $(a+1, b+1)/(a,b)$ and content $\sigma = (1,1)$.
    \item Suppose $\lambda = (a+1, b, 1)$ or $(a, b+1, 1)$. The tableau of shape $\lambda/(a,b)$ has two boxes, each on a distinct row. Therefore, there is a unique tableau of this shape with content $\sigma = (1,1)$.
    \item Suppose $\lambda = (a,b,2)$. The tableau of shape $(a, b, 2)/(a,b)$ has one row with two boxes, and the lattice word condition imposes that all boxes of the first row must be labeled $1$. Therefore, there are no such tableau with content $\sigma = (1,1)$.
    \item Suppose $\lambda = (a, b, 1, 1)$. The tableau of shape $\lambda/(a,b)$ has two boxes, in two rows and one column. Therefore, there is a unique tableau of shape $\lambda/(a,b)$ and content $\sigma = (1,1)$.
  \end{enumerate}
\end{enumerate}

\begin{figure}
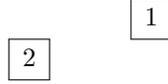

\begin{ytableau}
\none & \none & \none &  1 \\
2 
\end{ytableau} 
\caption{The unique Littlewood--Richardson tableaux of shape $(a, b)/(a-1, b-1)$ or $(a+1, b+1)/(a, b)$ and content $(1,1)$.}
\label{fig:LR-coeff}
\end{figure}

\begin{figure}
\begin{center}
\begin{tabular}{|c|c|c|c||c|c|c||c|} 
 \hline
 $\lambda$ & $\zeta$ & $\sigma$ & $\tau$ & $c_{\zeta\sigma}^{(1,1)}$ & $c_{\zeta\tau}^{(a,b)}$ & $c_{\sigma\tau}^{\lambda}$ & $\pi_{\Sp(4)}(\chi_{\Sp}(\lambda))$ \\ [0.5ex] 
 \hline\hline
 $(a-1, b-1)$ & $(1,1)$ & $(0,0)$ & $(a-1, b-1)$ & $1$ & $1$ & $1$ & $\chi_{\Sp(4)}(\lambda)$ \\ 
 \hline
 $(a-1, b+1)$ & $(1,0)$ & $(1,0)$ & $(a-1, b)$ & $1$ & $1$ & $1$ & $\chi_{\Sp(4)}(\lambda)$ \\
 \hline
 \multirow{2}{*}{$(a,b)$} & $(1,0)$ & $(1,0)$ & $(a-1, b)$ & $1$ & $1$ & $1$ & \multirow{2}{*}{$\chi_{\Sp(4)}(\lambda)$} \\
 & $(1,0)$ & $(1,0)$ & $(a, b-1)$ & $1$ & $1$ & $1$ &  \\
 \hline
 $(a+1, b-1)$ & $(1,0)$ & $(1,0)$ & $(a, b-1)$ & $1$ & $1$ & $1$ &$\chi_{\Sp(4)}(\lambda)$  \\
 \hline
 $(a+1, b+1)$ & $(0,0)$ & $(1,1)$ & $(a,b)$ & $1$ & $1$ & $1$ & $\chi_{\Sp(4)}(\lambda)$\\
 \hline
 $(a, b, 1, 1)$ & $(0,0)$ & $(1,1)$ & $(a, b)$ & $1$ & $1$ & $1$ & $-\chi_{\Sp(4)}(a,b)$ \\
 \hline
 $(a-1, b, 1)$ & $(1,0)$ & $(1,0)$ & $(a-1, b)$ & $1$ & $1$ & $1$ & $0$ \\ 
 \hline
 $(a, b-1, 1)$ & $(1,0)$ & $(1,0)$ & $(a, b-1)$ & $1$ & $1$ & $1$ & $0$ \\
 \hline
 $(a+1, b, 1)$ & $(0,0)$ & $(1,1)$ & $(a,b)$ & $1$ & $1$ & $1$ & $0$ \\ 
 \hline
 $(a, b+1, 1)$ & $(0,0)$ & $(1,1)$ & $(a,b)$ & $1$ & $1$ & $1$ & $0$ \\ 
 \hline
\end{tabular}
\end{center}
\caption{All partitions $\lambda$, $\zeta$, $\sigma$, and $\tau$ such that $c_{\zeta\sigma}^{(1,1)} c_{\zeta\tau}^{(a,b)} c_{\sigma\tau}^{\lambda} \neq 0$ for fixed $(a, b)$ with $a \geq b \geq 0$. The values given for $c_{\alpha\beta}^{\gamma}$ assume that $\alpha$, $\beta$, and $\gamma$ are all partitions. Otherwise, replace the nonzero value given for $c_{\alpha\beta}^{\gamma}$ with $0$.}\label{fig:LR-table}
\end{figure}
In Figure \ref{fig:LR-table}, we record the calculations of the Littlewood--Richardson coefficients $c_{\zeta\sigma}^{(1,1)}$, $c_{\zeta\tau}^{(a,b)}$, and $c_{\sigma\tau}^\lambda$ for partitions $\lambda$, $\zeta$, $\sigma$, and $\tau$ found above such that $c_{\zeta\sigma}^{(1,1)}c_{\zeta\tau}^{(a,b)}c_{\sigma\tau}^{\lambda} \neq 0$. Next, we compute $N_{(1,1)(a,b)\lambda} \chi_{\Sp}(\lambda)$ for some of the partitions $\lambda$ listed in Figure \ref{fig:LR-table}. For any subset $S\subseteq \mathscr P$, the indicator function of $S$ is denoted by $\ind_S$ (cf. Definition \ref{defn:indicator}). Recall that if $\lambda$ is not a partition, then $\chi_{\Sp}(\lambda) = 0$ by definition.
\begin{enumerate}
  \item If $\lambda = (a-1, b-1)$, then $\zeta = (1,1)$, $\sigma = (0,0)$, and $\tau = (a-1, b-1)$, so 
  \[
    N_{(1,1)(a,b)(a-1,b-1)} = c_{(1,1)(0,0)}^{(1,1)}c_{(1,1)(a-1,b-1)}^{(a,b)}c_{(0,0)(a-1,b-1)}^{(a-1,b-1)} = \ind_{b \geq 1}(a,b)
  \]
  and
  \[
    N_{(1,1)(a,b)(a-1,b-1)} \chi_{\Sp}(a-1,b-1) = \ind_{b \geq 1}(a,b) \chi_{\Sp}(a-1,b-1) = \chi_{\Sp}(a-1,b-1).
  \]

  \item If $\lambda = (a-1, b+1)$, then $\zeta = (1,0)$, $\sigma = (1,0)$, and $\tau = (a-1, b)$, so
  \[
    N_{(1,1)(a,b)(a-1, b+1)} = c_{(1,0)(1,0)}^{(1,1)} c_{(1,0)(a-1, b)}^{(a,b)} c_{(1,0)(a-1, b)}^{(a-1, b+1)}  = \ind_{a \geq b + 2}(a,b)
  \]
  and
  \[
    N_{(1,1)(a,b)(a-1,b+1)}\chi_{\Sp}(a-1,b+1) = \ind_{a \geq b + 2}(a,b) \chi_{\Sp}(a-1,b+1) = \chi_{\Sp}(a-1,b+1).
  \]

  \item If $\lambda = (a, b)$, then 
  \begin{enumerate}
    \item $\zeta = (1,0)$, $\sigma = (1,0)$, and $\tau = (a-1, b)$ or
    \item $\zeta = (1,0)$, $\sigma = (1,0)$, and $\tau = (a, b-1)$, so
  \end{enumerate}
  \[
    N_{(1,1)(a,b)(a,b)} = c_{(1,0)(1,0)}^{(1,1)} c_{(1,0)(a-1,b)}^{(a,b)} c_{(1,0)(a-1,b)}^{(a,b)} + c_{(1,0)(1,0)}^{(1,1)} c_{(1,0)(a, b-1)}^{(a,b)} c_{(1,0)(a, b-1)}^{(a,b)} = \ind_{a \geq b+1}(a, b) + \ind_{b \geq 1}(a,b)
  \]
  and
  \[
    N_{(1,1)(a,b)(a,b)}\chi_{\Sp}(a,b) = (\ind_{a \geq b+1}(a, b) + \ind_{b \geq 1}(a,b)) \chi_{\Sp}(a,b).
  \]

  \item If $\lambda = (a+1, b-1)$, then $\zeta = (1,0)$, $\sigma = (1,0)$, and $\tau = (a, b-1)$, so 
  \[
    N_{(1,1)(a,b)(a+1,b-1)} = c_{(1,0)(1,0)}^{(1,1)}c_{(1,0)(a,b-1)}^{(a,b)}c_{(1,0)(a,b-1)}^{(a+1,b-1)} = \ind_{b \geq 1}(a,b)
  \]
  and
  \[
    N_{(1,1)(a,b)(a+1,b-1)}\chi_{\Sp}(a+1,b-1) =  \ind_{b \geq 1}(a,b) \chi_{\Sp}(a+1,b-1) = \chi_{\Sp}(a+1,b-1).
  \]

  \item If $\lambda = (a+1, b+1)$, then $\zeta = (0,0)$, $\sigma = (1,1)$, and $\tau = (a,b)$, so
  \[
    N_{(1,1)(a,b)(a+1, b+1)} = c_{(0,0)(1,1)}^{(1,1)} c_{(0,0)(a,b)}^{(a,b)} c_{(1,1)(a,b)}^{(a+1, b+1)} = 1
  \]
  and
  \[
    N_{(1,1)(a,b)(a+1, b+1)}\chi_{\Sp}(a+1,b+1) = \chi_{\Sp}(a+1,b+1).
  \]

  \item If $\lambda = (a,b,1,1)$, then $\zeta = (0,0)$, $\sigma = (1,1)$, and $\tau = (a,b)$, so
  \[
    N_{(1,1)(a,b) (a,b,1,1)} = c_{(0,0)(1,1)}^{(1,1)} c_{(0,0)(a,b)}^{(a,b)} c_{(1,1)(a,b)}^{(a,b,1,1)} = \ind_{b \geq 1}(a,b)
  \]
  and
  \[
    N_{(1,1)(a,b) (a,b,1,1)} \chi_{\Sp}(a,b,1,1) = \ind_{b \geq 1}(a,b) \chi_{\Sp}(a,b,1,1).
  \]
\end{enumerate}
Aside from the six partitions $\lambda$ considered above, the remaining partitions $\lambda$ such that $N_{(1,1)(a,b)\lambda} \neq 0$ are $\lambda = (a+\varepsilon_1, b+\varepsilon_2, 1)$ for $(\varepsilon_1, \varepsilon_2) = (\pm 1, 0), (0, \pm 1)$. For such $\lambda = (a+\varepsilon_1, b+\varepsilon_2, 1)$, compute that
\[
  \lambda_i' = \begin{cases}
  3 & i = 1, \\
  2 & 2 \leq i \leq b+\varepsilon_2, \\
  1 & b + \varepsilon_2 < i \leq a + \varepsilon_1.
  \end{cases}
\]
With $i = 1$,
\[
  \lambda_i' - (i-1) =  3 = n+1
\]
which implies that $\pi_{\Sp(4)}(\chi_{\Sp}(a + \varepsilon_1, b + \varepsilon_2, 1)) = 0$.

Next, we determine $\pi_{\Sp(4)}(\chi_{\Sp}(a, b, 1, 1))$. Compute that 
\[
  \lambda_i' = \begin{cases}
  4 & i = 1, \\ 
  2 & 2 \leq i \leq b \\
  1 & b < i \leq a.
  \end{cases}
\]
Then $\lambda_i' - (i-1) \neq n+1 = 3$ for all $i$. Compute that
\begin{align*}
  k_i &= \begin{cases}
  2 & 1 \leq i \leq b \\
  1 & b < i \leq a.
  \end{cases}
\end{align*}
The values $t_i = k_i - (i-1)$ are all distinct and decreasing so compute that 
\[
  \mu_i' = t_i + (i-1) = k_i = \begin{cases}
  2 & 1 \leq i \leq b \\
  1 & b < i \leq a.
  \end{cases}
\]
The Young diagram defined by the values $\mu_i'$ is the diagram of the partition $(a,b)$. Since the $t_i$ are strictly decreasing, the permutation $\sigma$ reordering these terms is trivial. The only index $i$ such that $\lambda_i' - (i-1) > n = 2$ is $i = 1$, and so the number $s$ of such indices is $1$. Finally, \cite[Proposition 2.4.1]{koike--terada} implies that 
\[
  \pi_{\Sp(4)}(\chi_{\Sp}(a, b, 1, 1)) = (-1)^{\textup{sgn}(\sigma) + s} \chi_{\Sp(4)}(a,b) = -\chi_{\Sp(4)}(a,b).
\]
We record the results of these calculations of $\pi_{\Sp(4)}(\chi_{\Sp}(\lambda))$ in Figure \ref{fig:LR-table}. Recall that $\pi_{\Sp(4)}(\chi_{\Sp}(\lambda)) = \chi_{\Sp(4)}(\lambda)$ if $\ell(\lambda) \leq 2$ by \cite[Proposition 2.2.1]{koike--terada}.

Applying Theorem \ref{thm:tensors} and applying the specialization homomorphism $\pi_{\Sp(4)}$ with \cite[Proposition 2.2.1]{koike--terada} and the above computations,
\begin{align*}
\chi_{\Sp(4)}(1,1) \chi_{\Sp(4)}(a,b) &= \sum_{\lambda \in \mathscr P} N_{(1,1)(a,b)\lambda} \pi_{\Sp(4)} (\chi_{\Sp}(\lambda)) \\
&= (\ind_{a \geq b+1}(a,b) + \ind_{b \geq 1}(a,b)) \chi_{\Sp(4)}(a,b) + \chi_{\Sp(4)}(a-1, b-1) + \chi_{\Sp(4)} (a-1, b+1) \\
&\quad + \chi_{\Sp(4)}(a+1, b-1) + \chi_{\Sp(4)}(a+1, b+1) \\
&\quad + \left(\sum_{(\varepsilon_1, \varepsilon_2) = (\pm 1, 0), (0, \pm 1)}N_{(1,1)(a,b)(a+\varepsilon_1, b+\varepsilon_2, 1)}\pi_{\Sp(4)}(\chi_{\Sp}(a+\varepsilon_1, b + \varepsilon_2, 1))\right) \\
&\quad + \ind_{b \geq 1}(a,b)\pi_{\Sp(4)}(\chi_{\Sp}(a,b,1,1)) \\
&= (\ind_{a \geq b+1}(a,b) + \ind_{b \geq 1}(a,b)) \chi_{\Sp(4)}(a,b) + \chi_{\Sp(4)}(a-1, b-1) + \chi_{\Sp(4)} (a-1, b+1) \\
&\quad + \chi_{\Sp(4)}(a+1, b-1) + \chi_{\Sp(4)}(a+1, b+1) - \ind_{b \geq 1}(a,b)\chi_{\Sp(4)}(a,b) \\
&= \ind_{a \geq b+1}(a,b)\chi_{\Sp(4)}(a,b) + \chi_{\Sp(4)}(a-1, b-1) + \chi_{\Sp(4)} (a-1, b+1) \\
&\quad + \chi_{\Sp(4)}(a+1, b-1) + \chi_{\Sp(4)}(a+1, b+1).
\end{align*}

The character of the $\Sp(4, \Q_\ell)$-representation $W_{1,1} \otimes W_{a,b}$ is the product $\chi_{\Sp(4)}(1,1) \chi_{\Sp(4)}(a,b)$. Therefore, rewriting the above equality on the level of $\Sp(4, \Q_\ell)$-representations proves this lemma.
\end{proof}

\bibliographystyle{alpha}
\bibliography{universal-abelian-surface}
\end{document}